\documentclass[12pt]{amsart}
\usepackage{amsmath,amssymb,amsfonts,latexsym}
\usepackage{mathrsfs}
\usepackage[dvips]{geometry}
\usepackage{array}
\usepackage{epsf}
\usepackage{epsfig}
\newtheorem*{lemma}{Lemma}
\newtheorem*{prop}{Proposition}
\newtheorem*{thm}{Theorem}
\newtheorem*{cor}{Corollary}

\newcommand{\twoheaddownarrow}{\overset{\sim}{\twoheaddownarrow}}
\newcommand{\End}{\operatorname{End}}

\newcommand{\gr}{\operatorname{gr}}
\newcommand{\nc}{\newcommand}
\nc{\Ker}{\operatorname{Ker}} \nc{\rker}{\operatorname{rKer}}
\nc{\im}{\operatorname{Im}}
\nc{\stab}{\operatorname {Stab}}
\nc{\ann}{\operatorname {Ann}}
\nc{\Id}{\operatorname {Id}}
\nc{\Prim}{\operatorname {Prim}}
\nc{\Real}{\operatorname {Re}}
\nc{\Ext}{\operatorname {Ext}}
\nc{\rad}{\operatorname {rad}}

\usepackage{tikz-cd}
\usepackage[center]{caption}
\usepackage{float}

\begin{document}

\title [Weierstrass Sections]{Weierstrass Sections for Parabolic adjoint action in type $A$}
\author [Yasmine Fittouhi and Anthony Joseph]{Yasmine Fittouhi and Anthony Joseph\\
Donald Frey, Professional Chair\\
Department of Mathematics\\
The Weizmann Institute of Science\\
Rehovot, 76100, Israel}
\date{\today}
\maketitle


\


 \

\textbf{Abstract}.

\

The notion of ``Weierstrass Section'', comes from Weierstrass canonical form for elliptic curves.  In celebrated work [B. Kostant, Lie group representations on polynomial rings, Amer. J. Math. 85 (1963), 327–-404] constructed such a section for the action of a semisimple Lie algebra on its dual using a principal $s$-triple.  Actually it is enough to have an ``adapted pair'' and indeed the construction in [A. Joseph and D. Shafrir, Polynomiality of invariants, unimodularity and adapted pairs, Transform. Groups 15 (2010), no. 4, 851–-882] works rather well for the coadjoint action of an algebraic, but not necessarily reductive Lie algebra.

In the present work a Weierstrass section is constructed for the adjoint action of the derived algebra of a parabolic subalgebra on its nilradical in type $A$.  The starting point is Richardson's theorem which implies the polynomiality of the invariant subalgebra.  Here adapted pairs seldom exist.  A new construction is developed and this is mainly combinatorial based on joining boxes in the Young tableau associated to the ``Richardson component''.  Indications are given for extending this construction in other types.

The construction has relations to quivers [T. Br\"ustle, L. Hille, Lutz, C.M. Ringel and G. R\"ohrle, The $\Delta$-filtered modules without self-extensions for the Auslander algebra of $k[T]/⟨Tn⟩$. Algebr. Represent. Theory 2 (1999), no. 3, 295–-312] and to hypersurface orbital varieties  [A. Joseph and A. Melnikov, Quantization of hypersurface orbital varieties in $\mathfrak {sl}(n)$. The orbit method in geometry and physics (Marseille, 2000), 165–-196, Progr. Math., 213].

\section{Definitions}\label{1}

The ground field is assume to be the complex numbers $\mathbb C$.  For every positive integer $n$, set $[1,n]=\{1,2,,\ldots\,n\}$.

\subsection{Weierstrass Sections}\label {1.1}

Let $P$ be a connected algebraic group (not necessarily reductive) acting by morphisms on a vector space $M$, necessarily finite dimensional.  Let $\mathfrak p$ be the Lie algebra of $P$.

Given $\mathscr V$ an algebraic variety let $\mathbb C[\mathscr V]$ denote the algebra of regular functions on $\mathscr V$.

A linear subvariety $e+V$ of $M$ is by definition the translate of a vector subspace $V$ of $M$ by an element $e \in M$.  Let $\varphi$ denote the morphism of $\mathbb C[M]^P$ into $\mathbb C[e+V]$ defined by restriction.

Following Popov and Vinberg \cite {PV}, we define a Weierstrass section for the action of $P$ on $M$ to be a linear subvariety $e+V$  of $M$ such that $\varphi$ is an isomorphism.

The algebra $\mathbb C[e+V]$ identifies with the symmetric algebra on $V^*$ and so is polynomial.  However the polynomiality of the invariant algebra $\mathbb C[M]^P$ is not sufficient to ensure the existence of a Weierstrass section (\ref {1.2}).

Geometrically the existence of a Weierstrass section means that a $P$ orbit in $M$ meets $e+V$ in at most one point.

\subsection{Regularity and Adapted Pairs}\label {1.2}

\subsubsection {}\label {1.2.1}

A $P$ orbit in $M$ is said to be regular if it has the minimal codimension.  An element of $M$ is said to be regular if it lies in a regular orbit.  Obviously if $e+V$ meets only regular orbits, then $e$ must be regular.

\

\textbf{Definition.} An adapted pair for the action of $P$ on $M$ is a pair $e \in M, h \in \mathfrak g$ such that $e$ is regular and $h$ acts reductively on $M$ with $h.e=-e$.

\

For this notion to be useful for constructing Weierstrass sections, one needs to ``truncate'' $P$, so that it admits no proper semi-invariants for its action on $M$.  This truncation is canonical for $\mathfrak g$ being almost algebraic but can also be defined in general \cite [Prop. 3.1] {OV}.  At the level of the Lie algebra $\mathfrak p$ of $P$, it just means cutting down the Cartan subalgebra of $\mathfrak p$ so that it exactly vanishes on the weights of the semi-invariants. On the other hand this truncation may eliminate the possibility of finding an adapted pair.

\subsubsection {}\label {1.2.2}

In general adapted pairs are hard to find. If $G$ is a reductive group acting on the dual of its Lie algebra, then an adapted pair is given by the first two elements of a principal $s$-triple.  For the action of a truncated parabolic (or even biparabolic) subalgebra in type $A$ on the dual of its Lie algebra, adapted pairs were constructed in \cite {J2}.  Yet outside type $A$, the truncated Borel generally has not enough (if any) semisimple elements, so adapted pairs cannot exist.  Yet outside types $B_{2n},C,F$ the truncated Borel does admit a Weierstrass section \cite [Thm. 9.4]{J3}.  This section can have some unusual properties \cite [11.4, Example 3]{J3}.  Thus it may contain non-regular elements and $P(e+V)$ need not be open in $M$.

In the remaining cases even the existence of a Weierstrass section is in doubt.  As noted in  \cite [11.4, Example 2]{J3} in type $C_2$ one cannot obtain a Weierstrass section at all because the invariant generators take the form $x,xy+z^2$, with $x,y,z,t$ being root vectors. Thus if $e+V$ is a proposed Weierstrass section $x,z$ should not vanish on $V$. Then even if $e$ is non-zero on $y$, we are left with the quadratic element $z^2$. Geometrically any orbit non-vanishing on $z$ meets $x+V$ at two points.

\subsubsection {}\label {1.2.3}

Given an adapted pair, let $V$ be an $h$-stable complement to $\mathfrak p.e$ in $M$. Obviously $\dim V$ must be the codimension of $P.e$ and hence of any regular $P$ orbit in $M$.   By construction $P.(e+V)$ is dense in $M$.  In particular the restriction map $\varphi$ of \ref {1.1} is injective.

\subsubsection {}\label {1.2.4}

Assume in the remainder of section \ref {1.2} that the action of $P$ on $M$ admits no proper semi-invariants.  Then the invariants in the fraction field of $\mathbb C[M]$ is just the fraction field of $\mathbb C[M]^P$. By a theorem of Chevalley (often referred to as a theorem of Rosenlicht \cite {R}) the former has transcendence degree $\dim V$.  Consequently $\mathbb C[e+V]$ is algebraic over the image of $\mathbb C[M]$ in $\mathbb C[e+V]$.

\subsubsection {}\label {1.2.5}

Assume that $\varphi$ is surjective and let $v_i:i \in I$ be a basis of $V$ consisting of $h$ eigenvectors of eigenvalue $m_i$.  Then $\varphi^{-1}(v_i)$ is a set of homogeneous generators of $\mathbb C[M]^P$ of degree $d_i:=m_i+1$.

\subsubsection {}\label {1.2.6}

Suppose that $\mathbb C[M]^P$ is polynomial.  Then the above relation between degrees of generators and eigenvalues implies the surjectivity of $\varphi$.  For coadjoint action one can do better,  \cite [Sect. 6.2 and Thm. 6.3]{JS}, as this condition relating degrees and eigenvalues is not needed, but is rather a consequence.

\subsubsection {}\label {1.2.7}

The surjectivity of $\varphi$ implies that the eigenvalues of $h$ on $V$ are non-negative.   Since $e$ has $h$ eigenvalue $-1$ and is regular, it follows that every element of $e+V$ is regular by \cite [7.8]{J3}.  This argument is quite standard but a little tricky since it needs the base field to be infinite.

\

\subsection{The Nilfibre}\label {1.3}

Let $\mathbb C[M]^P_+$ denote the subalgebra of $\mathbb C[M]^P$ spanned by the homogeneous invariants of positive degree.

\subsubsection{}\label {1.3.1}

The nilfibre $\mathscr N$ for the action of $P$ on $M$ is the zero variety of ideal $J$ generated by $\mathbb C[M]^P_+$ in $\mathbb C[M]$.  One may remark that $e,h$ is an adapted pair, then $e \in \mathscr N$.

\subsubsection{}\label {1.3.2}

Since $\mathbb C[M]^P_+$ is a $P$ submodule of $\mathbb C[M]$, it follows that $\mathscr N$ is $P$ stable. Since $P$ is connected, the irreducible components of $\mathscr N$ are also $P$ stable.

\subsubsection{}\label {1.3.3}

Assume that $P$ has no proper semi-invariants in $\mathbb C[M]$.

Let $e+V$ be a Weierstrass section for the action of $P$ on $M$.

We noted in \ref {1.2.2}, that $e+V$ may admit non-regular elements.

Now assume that $e+V$ has non-regular elements.

It need not be the case that $e+V$ meets every regular orbit.  Many examples \cite {J5} arise from the coadjoint action of a parabolic in type $A$.

A sufficient condition is for $e+V$ to meet every regular orbit is for $J$ to be a prime ideal \cite [Cor. 8.7]{J4}, the argument being due to Kostant. Remarkably Kostant \cite {K} showed that $J$ is prime if $\mathscr N$ is irreducible.  For an exposition see \cite [Thm. 8.1.3]{D}.

\subsubsection{}\label {1.3.4}

The (counter)-examples in \ref {1.3.3} all came from parabolic coadjoint action.

Here we intimate a study for parabolic adjoint action.  The invariants for these action can be \textit{very} different.  Yet there are some remarkable connections especially concerning the weights of the invariant generators \ref {3.5} and \ref {3.6}.

\section{Richardson's Theorem and its Consequences}\label{2}

\subsection{Richardson's Theorem}\label {2.1}

From now on $P$ denotes a parabolic subgroup of a simple connected simply-connected Lie group $G$.  Let $\mathfrak g$ (resp. $\mathfrak p$) denote the Lie algebra of $G$ (resp. $P$).

 One may write $\mathfrak p$ as a direct sum of its Levi factor $\mathfrak r$ and its nilradical $\mathfrak m$.  From now on $M$ denotes $\mathfrak m$.

Fix a Cartan subalgebra $\mathfrak h$ for the Lie algebra $\mathfrak g$ of $G$.  Let $\pi$ denote a choice of simple roots with respect to $\mathfrak h$. Then $P$ may be chosen so that $\mathfrak h \subset \mathfrak p$ and the Levi factor of $\mathfrak p$ is determined by a subset $\pi'$ of $\pi$.  In this case we write $\mathfrak p_{\pi'}$ for $\mathfrak p$.  When $\pi'$ is empty, $\mathfrak p_{\pi'}$ identifies with a Borel subalgebra $\mathfrak b$ and we denote by $\mathfrak n$ its nilradical.

Richardson's theorem asserts that $P$ admits a dense open orbit in $\mathfrak m$.

 Let $P'$ denote the derived group of $P$. Let $\mathfrak p'$ denote its Lie algebra.  Since $P'$ admits no non-trivial semi-invariants, it follows from Chevalley's theorem that the minimal codimension of a $P'$ orbit in $\mathfrak m$ is the number of generators of the polynomial algebra $\mathbb C[\mathfrak m]^{P'}$.

\subsection{Polynomiality}\label {2.2}

\subsubsection{}\label {2.2.1}
 Via Chevalley's theorem noted in \ref {1.2.4},  Richardson's theorem is equivalent to the fraction field of $\mathbb C[\mathfrak m]$ admitting no non-scalar invariant functions.  This is in turn equivalent to $\mathbb C[\mathfrak m]^{\mathfrak p'}$ being multiplicity-free as an $\mathfrak h$ module.  This follows by a simple argument using unique factorisation.

 \subsubsection{}\label {2.2.2}

 Since $\mathfrak p'$ is generated by ad-nilpotent elements and the base field has zero characteristic, it follows that factors of weight vectors in $\mathbb C[\mathfrak m]^{\mathfrak p'}$ are again weight vectors.  Thus $\mathbb C[\mathfrak m]^{\mathfrak p'}$ is generated by its weight vectors which are irreducible polynomials.  Since $\mathbb C[\mathfrak m]$ is factorial, so is $\mathbb C[\mathfrak m]^{\mathfrak p'}$ and being multiplicity-free further implies that its weight vectors, which are irreducible polynomials, are algebraically independent.  (This argument in a slightly less general form first appeared in \cite [Sect. 4]{J0}.  It was also given in detail in \cite [1.7]{FJ}.)

 \subsubsection{}\label {2.2.3}

 In particular $\mathbb C[\mathfrak m]^{\mathfrak p'}$ is polynomial and one can ask how to describe its generators explicitly.

\subsection{Orbital Varieties}\label {2.3}
\subsubsection {} \label {2.3.1}
Let $\mathscr O$ be a co-adjoint $G$ orbit.  Via the Killing form we can identify $\mathfrak g^*$ with $\mathfrak g$.  We say that $\mathscr O$ is a nilpotent orbit if $\mathscr O \cap \mathfrak n$ is non-empty.  In this case we have
\cite [Remark following Eq. (1)]{S}, the important dimension formula
$$ \dim \mathscr O \cap \mathfrak n = \frac{1}{2} \dim \mathscr O. \eqno {(1)}$$

The proof results from the Steinberg triple variety whose construction is mainly based on Bruhat decomposition.

\subsubsection {} \label {2.3.2}

Equation $(1)$ is all that is needed to give a proof of Richardson's theorem in a few lines.  This is perhaps well-known but as it emphasizes the role of orbital varieties, we give a proof anyway.

\begin {cor} $P$ admits a dense orbit in $\mathfrak m$.
\end {cor}

\begin {proof}
After Dynkin, there are only finitely many nilpotent orbits. In particular $G\mathfrak m$ is a finite union of nilpotent orbits $\{ \mathscr O_i=Gx_i\}_{i=1}^k$.  Again $G$ and $\mathfrak m$ are connected.  Thus $G\mathfrak m$ is irreducible, so admits a unique dense orbit say $\mathscr O_1$.  Then
$$ \mathfrak m = \mathfrak m \cap G\mathfrak m = \cup_{i=1}^k(\mathfrak m \cap \mathscr O_i),$$
so by $(1)$ we obtain
$$\dim \mathfrak m = \max _i \dim (\mathfrak m\cap \mathscr O_i) \leq \max _i \dim (\mathfrak n\cap \mathscr O_i)= \frac{1}{2} \max _i \dim \mathscr O_i, \eqno {(2)}$$

Yet $\dim G/P= \dim \mathfrak m$ and $Px_i \subset \mathfrak m$. Consequently
$$\dim Gx_i \leq \dim G/P+\dim Px_i \leq 2\dim \mathfrak m, \forall i. \eqno {(3)}$$

So by a particular case of $(3)$ and by $(2)$ we obtain
$$\dim G\mathfrak m=\dim Gx_1\leq 2 \dim \mathfrak m\leq \max_i \dim \mathscr O_i=\dim G\mathfrak m, $$
so
$$\dim G\mathfrak m=2\dim \mathfrak m. \eqno {(4)}$$

  Again by $(3)$
$$\dim G\mathfrak m=\dim Gx_1\leq  \dim \mathfrak m +\dim Px_1.$$

Yet $Px_1 \subset \mathfrak m$ and so by $(4)$ we obtain
$$\dim Px_1=\dim \mathfrak m, \eqno {(5)}$$
as required.
\end {proof}

\subsubsection {}\label{2.3.3}

After Spaltenstein $ \dim \mathscr O \cap \mathfrak n$ is equidimensional \cite [Theorem]{S}.  Its irreducible components are called orbital varieties.  Again from the Steinberg triple variety it follows that $w \mapsto \overline {B(\mathfrak n \cap w \mathfrak n)}$ is a surjection of the Weyl group $W$ onto the set of orbital variety closures.  A proof can be found in \cite [Thm. 2.1.]{J01}

Richardson's theorem is closely associated with orbital varieties and as we have already noted, a proof derives from $(1)$.

\subsubsection {}\label{2.3.4}

Let $\mathfrak m$ be the nilradical of a parabolic.

\

\textbf{Definition.}  A hypersurface orbital variety in $\mathfrak m$, is an orbital variety closure which is a hypersurface in $\mathfrak m$.

\

 As noted in \cite [3.11]{JM} the ideal of definition of a hypersurface orbital variety is given by an irreducible polynomial which is $\mathfrak p$ semi-invariant.  Conversely the same analysis shows that an irreducible polynomial which is $\mathfrak p$ semi-invariant has as its zero variety a hypersurface orbital variety.

This gives the following result.  Recall the notation of \ref {2.1}.  In particular let $\pi'$  be the subset of $\pi$ defining $\mathfrak p_{\pi'}$ and hence its nilradical $\mathfrak m_{\pi'}$.

\begin {lemma} The number $N_{\pi,\pi'}$ of invariant (polynomial) generators of $S(\mathfrak m_{\pi'}^*)^{\mathfrak p'_{\pi'}}$ is just the number of hypersurface orbital varieties in $\mathfrak m_{\pi'}$.
\end {lemma}

\subsubsection {}\label{2.3.5}

In principle, the above lemma settles the
question posed in \ref {2.2.3}.  The only snag is that hypersurface orbital varieties are notoriously difficult to describe and to handle.  In type $A$ they were describe by Melnikov, the result being described in the rather long \cite [Sect. 2]{JM}, based on the lowering of numbered boxes inspired by the Robinson-Schensted algorithm.

From the parametrisation of orbital varieties by McGovern using Garfinkel domino tableaux, Perelman \cite {Pe} was able to describe hypersurface orbital varieties for the remaining classical types analogous to Melnikov's work for type $A$. Then in a tour de force she extended our results in \cite {JM} for $\mathfrak g$ classical. Yet it is extremely complicated even to state the results.  Again although she had some partial unpublished results on $E_6,F_4$, the remaining exceptional cases seem out of reach.

\section{Invariants}\label{3}

\subsection{The B-S Invariants}\label {3.1}

Concerning hypersurface orbital varieties in type $A$,  Benlolo and Sanderson \cite {BS} started the ball rolling by conjecturing the form of the required invariants in type $A$.  Surprisingly not noticed before, their presentation can be significantly simplified and then presented for all types.  This is explained in Sections \ref {3.2} and \ref {3.6}.

\subsection{A Reduction}\label{3.2}

Let $I$ be an index set for the set of connected components of $\pi'$.  Let $J \subset I$ be a proper subset. Let $\pi_J$ be a set of roots containing $\cup _{j \in J}\pi_j$ in $\pi$ and orthogonal to the roots in $\cup_{i \in I \setminus J}\pi_i$.

The containment   $\pi_J\supset\cup_{j\in J} \pi_j$ may be strict. On the other hand $\pi_J$ is strictly contained in $\pi\setminus \cup_{i \in I \setminus J}\pi_i$.

Let $\mathfrak g_{\pi_J}$ be the (proper) subalgebra of $\mathfrak g$ with simple root set $\pi_J$.  Let $\mathfrak p_J$ be the parabolic subalgebra of $\mathfrak g_{\pi_J}$ defined by $\cup_{j\in J} \pi_j=\pi_J\cap \pi'$ and let $\mathfrak m_J$ denote its nilradical.

Notice that for any vector space $M$ we may identify $\mathbb C[M]$ with the symmetric algebra $S(M^*)$ of $M^*$.

\begin {lemma}  One has $S(\mathfrak m^*_J)^{\mathfrak p_J'} \subset S(\mathfrak m^*)^{\mathfrak p'}$.
\end {lemma}

\begin {proof} Let $x_{-\beta}$ be a root vector in $\mathfrak m^*_J$. A simple root $\alpha \in \pi \setminus \pi_J$, cannot be in the support of $\beta$.  Hence $[x_\alpha,x_{-\beta}]=0$ and so $[x_\alpha,\mathfrak m^*_J]=0$.

 Again if $\alpha \in \cup_{i\in I\setminus J}\pi_i$, then $[x_{-\alpha},\mathfrak m^*_J]=0$, since the elements in $\cup_{i \in I\setminus J}\pi_i$  are orthogonal by assumption to those in $\pi_J$.

 Combined these prove the lemma.

\end {proof}

\textbf{Example}. Adopt the Bourbaki \cite {Bo} notation.  Suppose $\pi'=\pi \setminus \{\alpha_2\}$.  Then we can take $\pi_J=\{\alpha_1,\alpha_3\}$.

\subsection{Supplementary Generators}\label{3.3}

We call a supplementary generator for the pair $(\mathfrak g, \mathfrak p)$, a generator not lying in some $S(\mathfrak m^*_J)^{\mathfrak p_J'}$ obtained with $J$ defined as in \ref {1.3} above. We believe that the supplementary generators take a very special form, as suggested by type $A$.  This is discussed briefly in Sections \ref {3.5}, \ref {3.6}.

The supplementary generators already illustrate some classical results in invariant theory.  Thus let $n$ be a positive integer.  Take $\mathfrak g$ to be $n$ copies of $\mathfrak {sl}(2)$ acting on the (outer) $n$-fold tensor product $M$ of its the simple two dimensional $\mathfrak {sl}(2)$ module.  As is well-known $\mathbb C[M]^\mathfrak g$ is reduced to scalars if $n=1$, polynomial on one generator (a $2 \times 2$ minor) if $n=2$ and polynomial on one generator if $n=3$. These are supplementary generators when respectively $\pi'=\pi \setminus \{\alpha_2\}$ in type $A_3$ and when $\pi,\pi'$ are as in the example of \ref {3.2}.  For $n>3$, the invariant algebra is no longer polynomial and there is at the same time no choice of the pair $\pi,\pi'$ which via the Richardson theorem would predict polynomiality.  In addition we remark that for $n=3$, the generator is of degree $4$ on $8$ generators and in our opinion most easily computed using the technique described in \ref {3.6.2}.  It degenerates to a supplementary generator in type $G_2$, by identification of the variables in each of four pairs.


\subsection{Supplementary Generators in Type $A$}\label{3.4}
\subsubsection {}\label{3.4.1}

Recall that in type $A_{n-1}$, that is to say for $\mathfrak {sl}(n)$, the set of conjugacy classes of parabolic subgroups is given by the set of all sequences $n_1,n_2,\ldots,n_r$, where the $n_i$ are positive integers summing to $n$ (but not necessarily decreasing).  The Levi factor $\mathfrak r$ of $\mathfrak p$ is given as sequence of blocks of size $\{n_i\}_{i=1}^r$ going down the diagonal.

\subsubsection {}\label{3.4.2}

As is well-known the Levi factors of different parabolics  are conjugated by an (easily calculable)  element of $W$ if and only if the corresponding sets $n_i:i \in I$ coincide up to ordering - that is define the same partition of $n$.

\subsubsection {}\label{3.4.3}

One may remark that this element of the Weyl group does  \textit{not} leave the nilradical $\mathfrak m$ invariant. Moreover the invariant generators in $S(\mathfrak m^*)$ can be very different. Yet the description of the hypersurface orbital varieties in type $A$ by Melnikov combined with Lemma \ref {2.3.4}, shows that the number of polynomial generators in $S(\mathfrak m^*)$ is the same.

\subsubsection {}\label{3.4.4}

The above very remarkable result has the following explanation which carries over for all $\mathfrak g$ semisimple.

Let $\mathscr V$ be a closed $\mathfrak h$ stable subvariety of $\mathfrak n^+$, for example an orbital variety closure. Recall \cite [2.4,2.5] {J00} that the growth rate of $\mathbb C[\mathscr V]$  relative to an element of $\mathfrak h$ defines a function $r_\mathscr V$ on $\mathfrak h$.

When the $\mathscr V$ are the irreducible components of some $\mathfrak n \cap \mathscr  O$, these functions have some remarkable properties.  They are linearly independent and span a simple module $M_\mathscr O$ of $W$, which gives a concrete realisation of that part of the Springer representation corresponding to $\mathscr O$ together with the trivial representation of certain stabilizers of the component group of the centralizer in $G$ of any element in $\mathscr O$.  (One may remark that Springer obtained his representations via etale cohomlogy of fixed point sets in the flag variety.) This follows from the truth of \cite [Conj. 9.8]{J00} which was first established by Hotta \cite {H} and also in \cite [5.8]{J000}, a work which followed closely ideas of W. Rossmann using orbital integrals.   We remark that for type $A$ the proof does not need the high-brow luxury of etale cohomology or orbital integrals.  Indeed it follows from the fact the function $\Pi'$ (defined in \ref {3.4.5} below) generates the Specht module, which is simple and all Specht modules are obtained exactly once taking parabolics with non-conjugate Levi factors.

Let $\Delta$ (resp. $\Delta^+$) be the set of non-zero (resp. positive) roots and set $\Pi:=\prod_{\alpha \in \Delta^+} \alpha$.  Then we can write $r_\mathscr V=p_\mathscr V/\Pi$, where $p_\mathscr V$ is a homogeneous polynomial on $\mathfrak h$, whose degree equals $|\Delta^+|-\dim \mathscr V$.  We call $p_\mathscr V$ the characteristic polynomial of $\mathscr V$.

This assertion follows from \cite [Cor. 2.4]{J00} which easily generalises to give in particular.

\begin {lemma}  Let $\mathfrak  m$ be a subspace of $\mathfrak n$ spanned by a subset $x_\alpha:\alpha \in \Pi_\mathfrak m$ of root vectors. Let  $\mathscr V$ be a closed $\mathfrak h$ stable subvariety of $\mathfrak m$.  Then $r_\mathscr V=s_\mathscr V/\prod_{\alpha \in \Pi_\mathfrak m}\alpha$, where $s_\mathscr V$ is a homogeneous polynomial whose degree is the codimension of $\mathscr V$ in $\mathfrak m$.
\end {lemma}

\subsubsection {}\label{3.4.5}

Now set $\Delta^{\prime +}=\Delta^+\cap \mathbb N\pi'$,   $\Pi':=\prod_{\alpha \in \Delta^{\prime +}}\alpha$
and $\Pi'_-:=\prod_{\alpha \in \Delta^+\setminus \Delta^{\prime+}}\alpha=\Pi/\Pi'$.

Recall that $\mathfrak m _{\pi'}$ is an orbital variety closure, called the Richardson component.  Moreover  $p_{\mathfrak m_{\pi'}}=\Pi'$, equivalently $q_{\mathfrak m_{\pi'}}=1/\Pi'_-$.

\subsubsection {}\label{3.4.6}

\begin {lemma} Let $\mathscr V$ be a closed irreducible $\mathfrak h$ stable subvariety of $\mathfrak n$ and let $r_\mathscr V$ be its characteristic function.  Then $r_\mathscr V \Pi'_-$ is polynomial
 if and only if $\mathscr V \subset \mathfrak m_{\pi'}$.
\end {lemma}

Sufficiency follows from Lemma \ref {3.4.4}.

Necessity follows from \cite [Cor. 8.3]{J000}.  In detail suppose that $r_\mathscr V \Pi'_-$ is polynomial.  This means that if we write $r_\mathscr V=p_\mathscr V/\Pi$, then

\

$(*)$. The roots in $\Delta^{\prime +}$ divide $p_\mathscr V$.

\

Now \cite [Cor. 8.3]{J000} states that $p_\mathscr V$ is a sum with non-negative integer coefficients of a product of the form
$\prod_{i|C_i=C_{i+1}}\alpha_i$, where the $\alpha_i$ are \textit{pairwise distinct positive roots}.

In the above, the notation is as follows.  Fix an ordering $\{\alpha_1,\alpha_2\ldots, \alpha_r\}$ of the positive roots.  Let $\mathfrak m_i$ be the uniquely determined $\mathfrak h$ stable complement to $\mathbb C x_{\alpha_i}$ in $\mathfrak n$.  Define a chain of closed irreducible $\mathfrak h$ stable subvarieties $\{C_i\}_{i=1}^r$ of` $\mathfrak n$ inductively by setting $C_1=\mathscr V$ and taking $C_{i+1}$ to be an irreducible component of $C_i\cap \mathfrak m_i$.   The above sum is over all chains.  By Krull's theorem $C_{i+1}=C_i$ or has codimension $1$ in $C_i$.  Thus the above sum is a homogeneous polynomial whose degree is the codimension of $\mathscr V$ in $\mathfrak n$.

Now given a positive root $\alpha$, let $|\alpha|$ denote the sum of the coefficients when written as a sum of simple roots.  It is called the order of $\alpha$. Notice that the roots in the products can be taken in any order \cite [First line of 8.3]{J000}. We fix them by first taking then to belong to $\Delta^{\prime +}$ and then so that $\alpha \mapsto |\alpha|$ is increasing.

We claim that

\

\ $(*)$ implies that $\alpha \in \Delta^{\prime +}$ divides every factor in the above sum describing $p_\mathscr V$.

\

Indeed this is already clear for a simple root since such a root cannot be written as a non-trivial sum of positive roots with non-negative coefficients. These simple roots can then be discarded from each factor.

Now fix a positive integer $n$ and assume we have shown that for all $m \in [1,n]$ that we can discard every root $\alpha \in \Delta^{\prime +}$ of order $\leq m$, from every factor occurring in the above sum. Take $\alpha \in \Delta^{\prime +}$ of order $m+1$. Then it can only be written non-trivially as positive roots with positive coefficients if these roots lie in $\Delta^{\prime +}$ and hence are of order $\leq m$.  However all such roots have been discarded and so our hypothesis is extended to $m+1$.  This proves our claim.

The condition $C_i=C_{i+1}$  means that $C_i\cap \mathfrak m_i=C_i$ equivalently that $C_i \subset \mathfrak m_i$.

We conclude from the claim and the formula in  \cite [Cor. 8.3]{J000}, that $\mathscr V \subset \cap_{\alpha \in \Delta^{\prime +}} \mathbb C\mathfrak m_\alpha$. Yet the right hand side is just $\mathfrak m_{\pi'}$, proving the converse assertion.

\subsubsection {}\label{3.4.7}

Fix an orbital variety closure $\mathscr V$ and a simple root $\alpha$.  Then either $\mathscr V \subset \mathfrak m_\alpha$ or this fails.  In the first case $s_\alpha p_\mathscr V= - p_\mathscr V$.  In the second case $(s_\alpha +1)p_\mathscr V$ is a sum with non-negative integer coefficients of characteristic polynomials $p_{\mathscr V_i}$ of orbital varieties closures lying in $\mathfrak m_\alpha$ \cite [3.1,3.2]{J00}.

The above result has the following

\begin {cor}  Let $p$ be a sum with non-zero integer coefficients $c_i$ (not necessarily positive) of characteristic polynomials $\{p_i\}_{i=1}^k$ of pairwise distinct orbital varieties closures $\mathscr V_i$ and $\pi'$ a subset of $\pi$.  Suppose $s_\alpha p = -p$, for all $\alpha \in \pi'$.  Then $s_\alpha p_i=-p_i$, for all $\alpha \in \pi'$ and all $i \in [1,k]$.
\end {cor}

\begin {proof}  We can assume $\pi'$ non-empty otherwise there is nothing to prove and we can also consider simple roots individually.

Thus take $\alpha \in \pi'$.  Let $I$ be the subset $\{i \in [1,k]|\mathscr V_i \nsubseteq \mathfrak m_\alpha\}$ of $[1,k]$.  Then $2\sum_{i \in I}c_ip_i$ is a sum of characteristic polynomials of orbital varieties closures lying in $\mathfrak m_\alpha$ and therefore distinct from the $\mathscr V_i:i \in I$.  Then by linear independence (\ref {3.4.4}) it follows that $I$ is the empty set, as required.
\end {proof}


\subsubsection {}\label{3.4.8}

Let $\pi_1,\pi_2$ be subsets of $\pi$ and let $\mathfrak p_{\pi_1},\mathfrak p_{\pi_2}$ be the corresponding parabolics.  Set $\Pi^i=\prod_{\alpha \in \mathbb N^+\pi_i\cap \Delta^+}$, $\Pi^i_-=\Pi^i/\Pi$, for $i=1,2$.

\begin {prop} Suppose that the Levi factors of $\mathfrak p_{\pi_1},\mathfrak p_{\pi_2}$ are conjugated by some $w \in W$, equivalently $w\{s_\alpha\}_{\alpha \in \pi_1}=\{s_\alpha\}_{\alpha \in \pi_2}$. Then their nilradicals $\mathfrak m_{\pi_1},\mathfrak m_{\pi_2}$ admit the same number of hypersurface orbital varieties.
\end {prop}

\begin {proof} Recall the notation of \ref {3.4.5}.  Under the hypothesis $w\Pi^1_-=\Pi^2_-$, up to a sign.  Let $\mathscr V_j$ be a hypersurface orbital variety in $\mathfrak m_{\pi_1}$.  Its characteristic polynomial $p_{\mathscr V_j}$, or simply $p$,  takes the form $\varpi_j\Pi_1$, where  $\varpi_j$ is the weight of the generator of the ideal of definition of $\mathscr V_j$. The latter is $\mathfrak h'$ invariant and so $\varpi_j$ lies in the orthogonal of $\pi_1$.  Consequently $s_\alpha p=-p$, for all $\alpha \in \pi_1$.


By \ref {3.4.4}, $p':=w q_{\mathscr V_j}$ is  a sum of characteristic polynomials $\{p_k\}_{k\in K}$ of uniquely determined orbital variety closures. On the other hand $s_\alpha p'=-p'$, for all $\alpha \in \pi_2$. By Corollary \ref {3.4.7} it follows that $s_\alpha p_k=-p_k$, for all $\alpha \in \pi_2$ and all $k \in K$. Thus for all $k \in K$, it follows that $\Pi^2$ divides $p_k$, and so $r_{\mathscr V_k}\Pi^2_-$ is a polynomial.  Moreover since the action of $W$ does not change degree, it has the same degree as $\varpi_j$ which is one.

We conclude by Lemma \ref {3.4.6} that $\mathscr V_k$ is a subvariety of $\mathfrak m_{\pi_2}$ of codimension one in $\mathfrak m_{\pi_2}$.  Therefore it is a hypersurface orbital variety in
$\mathfrak m_{\pi_2}$, so its charateristic function takes the form $\varpi_k \Pi^2_-$ and belongs to the orthogonal of $\pi_2$ in $\mathfrak h$.


Thus may write  $\varpi_j=\sum_{k\in K} A_{j,k}\varpi_k$, for some matrix $A_{j,k} \in \mathbb C$ of coefficients.


Repeating the same argument with $w^{-1}$ it follows that the matrix $\{A_{j,k}\}$ is invertible.  This establishes a bijection between the hypersurface orbital varieties in $\mathfrak m_{\pi_1},\mathfrak m_{\pi_2}$, proving the proposition.
\end {proof}

\subsubsection {}\label{3.4.9}

The above result falls short of what we would like to prove, namely that $w\varpi_j/\Pi^2_-$ is a characteristic function of a \textit{single} hypersurface orbital variety.  In type $A$ this is established by explicit computation (\ref {4.6}).

By contrast there seems to be no relation between their degrees of the generators of the hypersurfaces in $\mathfrak m_{\pi_1},\mathfrak m_{\pi_2}$.  Indeed even the degree sums can be different.  Again there seems to be no relation between the Weierstrass sections in $\mathfrak m_{\pi_1}$ and in $\mathfrak m_{\pi_2}$.

\

\textbf{Example}.  Take $\pi =\{\alpha, \beta\}$ of type $B_2$ with $\beta$ the long root.  Then $\mathfrak m_\alpha$ admits a hypersurface orbital variety $\mathscr V$ and $p_\mathscr V=\alpha(\alpha+\beta)$, yet $\mathfrak m_\beta$ does not. Here the corresponding Levi factors are \textit{not} conjugate.  Again the minimal non-zero orbit meets $\mathfrak m_\alpha$ in codimension $2$, but $\mathfrak m_\beta$ in codimension $4$.

\subsubsection {}\label{3.4.10}

Recall the notation of \ref {3.4.1}. We obtain a supplementary invariant in type $A$ exactly when $n_1=n_r$ and $n_i\neq n_1$, for all $i \in [2,r-1]$.  It was described by Benlolo-Sanderson in \cite {BS} who calculated its degree to be
$$\sum_{i=1}^{r-1} \min {(n_i,n_1)}.\eqno {(7)}$$

This invariant is described precisely in \ref {4.1.6}.  We call it the Benlolo-Sanderson (supplementary) invariant.

\subsection{Kostant Weights}\label{3.5}

Let $\{\varpi_i\}$ be the set of fundamental weights and $w_0$ the unique longest element of the Weyl group.  Take $i \leq [\frac{n}{2}]$.  The $i^{th}$ Kostant weight (in type $A$ is defined to be $\varpi_i-w_o\varpi_i$.  It was shown by Dixmier \cite {D1} that $Y( \mathfrak n):=S(\mathfrak n)^\mathfrak n$ is a polynomial algebra of generators whose weights are Kostant weights. This was put in more general setting by Kostant who introduced what is now known as the Kostant cascade.  For every simple Lie algebra there is an element $f_i \in Y(\mathfrak n)$ whose weight is a Kostant weight and this element can be constructed from the Hopf dual of the enveloping algebra.  Nevertheless he missed \cite {K1} the fact that outside types $A,C$ $f_i$ may admit a polynomial square root which is a (polynomial) generator of $Y(\mathfrak n)$ and whose weight is half a Kostant weight.  These factors of $\frac{1}{2}$ are completely known \cite [5.4]{J6} but still somewhat mysterious.  We call them the coadjoint factors.  They only appear outside types $A$ and $C$.

\subsection{The highest degree Invariant Generator}\label{3.6}
\subsubsection{}\label {3.6.1}

The invariant algebra $S(\mathfrak g)^\mathfrak g$ is polynomial. It admits a highest degree invariant generator $z$ (of degree the number of positive roots plus one) which is algebraically independent of the remaining generators.  (It can be uniquely determined up to a scalar by the requirement that all its derivatives are harmonic polynomials, though it is not so clear if this uniqueness of great importance.) It type $A_{n-1}$ it is simply the determinant.  A similar description of $z$ using a determinant is possible for all $\mathfrak g$ classical.

\subsubsection{}\label {3.6.2}

Now take a generator $f_i$ of $Y(\mathfrak n)$.   Its weight $\varpi'_i$ is a Kostant weight (or half a Kostant weight). The generator $f_i$ is a semi-invariant for the parabolic subalgebra defined by the orthogonal $\pi^i$ of $\pi$ with respect to the Kostant weight. By passing to the dual we may view $f_i$ as a differential operator $f_i^*$.  Now take the derivative of $z$ with respect to $f_i^*$.  The resulting element has weight $-\varpi'_i$  is invariant for the opposed algebra $\mathfrak p^-_{\pi^i}$.

In type $A$ the above element is simply the lower left hand corner $n-i\times n-i$ minor.  It is a $\mathfrak p^-_{\pi^i}$ invariant.

\

\textbf{Example.}  Consider the example of \ref {3.2}.  The single invariant generator for the (coadjoint) action of $\mathfrak p'$ on $\mathfrak m^*$ is the highest root vector and its weight is \textit{half} of the  Kostant weight $\varpi_2-w_0\varpi_2$, whilst the single invariant generator for the action of $\mathfrak p'$ on $\mathfrak p$ is the element in \ref {3.3} with $n=3$.  It has degree $4$ and its weight is the minus the Kostant weight $\varpi_2-w_0\varpi_2$.

\subsubsection{}\label {3.6.3}

Now suppose that $\mathfrak g$ is of type $A_{n-1}$ and let $\mathfrak p$ be \textit{any} parabolic subalgebra given by blocks as above with $m=n_1=n_r$.  Let $\mathfrak m$ be its nilradical.  Restrict $f_m^*z$ to $\mathfrak m + \Id$.  It is an inhomogeneous polynomial.  Its leading term has degree given by $(2)$.  Moreover the leading term is a $\mathfrak p'$ invariant and exactly coincides with the Benlolo-Sanderson invariant.  The proof of its invariance was observed by Benlolo and Sanderson and can also be extracted from \cite [3.11]{JM}.  The result is entirely elementary and can be proved (with some attention to details) without using orbital varieties. Surprisingly the lower order terms are in general \textit{not} invariant.  It is also possible to show that this invariant is irreducible if and only if there is no $j: 1<j<r$ with $n_j=m$ \cite  [Prop. 3.15]{JM}. We give a new proof of irreducibility in \ref {5.3}.

\subsubsection{}\label {3.6.4}

In principle we thus have a way to describe supplementary invariants applicable to all simple Lie algebras.  There are of course some difficulties.  The first is to obtain a good description of $z$.  The second is to be able to match the parabolic with a Kostant weight.  Surprisingly further square roots can appear for example in type $C_3$. Consequently their weights can be half a Kostant weight.  We call these the adjoint factors. Outside type $A$ they can be non-trivial.  They differ from the coadjoint factors.  They are unknown in general; but in classical type can be read off from the results of Perelman \cite {Pe}.

\subsubsection{}\label {3.6.5}

The main difficulty of our procedure is show that we obtain \textit{all} the supplementary invariants.  This is relatively easy in type $A$ due to the rather explicit description of the hypersurface orbital varieties given in \cite [Sect. 2]{JM}. In classical type Perelman \cite {Pe} resorted to simply comparing the number of hypersurface orbital varieties given by domino manipulation with the number of invariants she could find.

\subsubsection{}\label {3.6.6}

Our proposal to overcome the above difficulty is to explicitly describe a regular $P'$ orbit in $\mathfrak m$.  Here we apply this procedure in type $A_{n-1}$.

Actually we do much more.  Namely we construct a Weierstrass section for the action of $P'$ on $\mathfrak m$, which is an entirely new result. Unlike the co-adjoint case \cite {J5}, it is not always possible to find an adapted pair.

\subsubsection{}\label {3.6.7}

Besides the above we also construct from our Weierstrass section a $P$ orbit in an irreducible component of $\mathscr N$ of dimension $\dim \mathfrak m - N_{\pi,\pi'}$.  One can ask if $\mathscr N$ is equi-dimensional (so then each component has dimension $\dim \mathfrak m - N_{\pi,\pi'}$).  Then such an orbit would be dense and one can ask if this sets up a bijection between suitably defined inequivalent Weierstrass sections and components of $\mathscr N$.

Behind our construction is the conjecture that $\mathscr N$ is equidimensional and each component admits a dense $P$ orbit $P.\hat{e}$ and moreover there exists $h \in \mathfrak h$ such that $h.\hat{e} = -\hat{e}$. Fix such a dense $P$ orbit and let $V$ be an $h$ stable vector space complement to $\mathfrak p.\hat{e}$.  We further conjecture that $\hat{e}+V$ is a Weierstrass section for the action of $P'$ on $\mathfrak m$ and that up to equivalence every Weierstrass section so obtains.  In type $A$ we give some examples of $\hat {e}$.

The advantage of this formulation is that it makes sense for a parabolic subalgebra of any simple Lie algebra.

\section{Construction of Weierstrass Sections in type $A$}\label{4}

\subsection {Numbered Tableaux} \label{4.1}

\subsubsection {} \label {4.1.1}

Fix a parabolic subalgebra $\mathfrak p$ in type $A_{n-1}$ and recall the notation of \ref {3.4.1}.   Set $n^i=\sum_{j=1}^{i-1}n_j$.

We shall represent this unordered sequence (in a standard fashion) as a union of columns $C_v:v=1,2,\ldots,r$, labelled by starting from the left with $C_v$ of height $n_v$.   We call this array, the diagram $\mathscr D$ defined by $\mathfrak p$. The rows of $\mathscr D$ are designated by $R_u:u=1,2,\ldots$ starting from above.   Each column of height $s$ is viewed as consisting of a stack of $s$ boxes.  The box lying on the intersection of $R_u$ and $C_v$ is labelled by $b_{u,v}$.
The height of $\mathscr D$ is defined to be $\max_{i=1}^r \{n_i\}$ and denoted by $ht \mathscr D$.

\subsubsection {} \label {4.1.2}

Two columns are said to be neighbouring (of height $s$) if they are of height $s$ and there are no columns of height $s$ between them. By \cite [Thm. 2.19]{JM} the number of hypersurface orbital varieties (and hence by Lemma \ref {2.3.4} the number of invariant generators of $S(\mathfrak m^*)^{\mathfrak p'}$) is the number of pairs of neighbouring columns.  However we shall not assume this result and indeed we shall give another proof (Corollary \ref {5.2.6}).

Through the reduction in \ref {3.4} we obtain a Benlolo-Sanderson invariant for every pair of neighbouring columns in $\mathscr D$.

\subsubsection {} \label {4.1.3}

Associate to neighbouring columns $C_v,C_{v'}$ of height $s$, the minor $M_{v,v'}$, or simply $M$, bounded by the corresponding blocks $B_v,B_{v'}$ of $\mathfrak r$.  More precisely the top row (resp. right hand column) of $M_{v,v'}$ is the top row of $B_v$ (resp. right hand column of $B_{v'}$), whilst the bottom row (resp. left hand column) of $M_{v,v'}$ is just above the top row of $B_{v'}$ (resp. just to the right of the right hand column of $B_v$).  Set $m_{v,v'}=\sum_{i=v+1}^{v'}n_i$, or simply, $m$. It is clear that $M$ is an $m\times m$ minor in $\End \mathbb C^n$.  It may be defined by its subset of rows $R_i:i\in I:=[n^v+1,n^{v'}]$ and columns $C_j:j \in J:= [s+n^v+1,s+n^{v'}]$ it meets.

\textbf{N.B.} If this construction is applied to columns of the same height $s$ which are not neighbouring, then we obtain a product of Benlolo-Sanderson invariants coming from the successive pairs of neighbouring columns of height $s$.  On the other hand as we note in \ref {5.3} the Benlolo-Sanderson invariants are themselves irreducible.

\subsubsection {} \label {4.1.4}

Insert the numbers $1,2,\ldots,n$, into $\mathscr D$ by placing them in the boxes successively down the columns and going from left to right. More precisely $u+n^v$ is placed in $b_{u,v}$. This is exactly the prescription given in \cite [2.12]{JM} as a first step in describing $\mathfrak m$ as an orbital variety closure using the Robinson-Schensted map.

$\mathscr D$  numbered in this fashion is called a tableau and denoted by $\mathscr T$.


\subsubsection {} \label {4.1.5}

The meaning of the construction of \ref {4.1.4} in the present context is the following. Let $x_{i,j}$ denote a standard matrix unit.  Then $x_{i,j} \in \mathfrak m$, if and only if $i\in \mathscr T$ lies in a column strictly to the left of $j\in \mathscr T$.


\

The above observation has the following easy though important consequence.


\begin {lemma} Suppose that $b_{u,v},b_{u',v'}:v<v'$ lie in that part of $\mathscr T$ bounded by a pair of neighbouring columns $C_v,C_{v'}$. Let $i$ (resp. $j$) be the entry of $b_{u,v}$ (resp. $b_{u',v'}$). Then $x_{i,j}\in M_{v,v'} \subset \mathfrak m$ and every element of $M_{u,v}$ so appears.
\end {lemma}

\subsubsection {} \label {4.1.6}

Actually it is more appropriate (see \ref {3.6.2}) to replace $M$ by its transpose $M^t$ viewed through the Killing form as a function on $\mathfrak g$ and by restriction as a function on $\mathfrak m$.

This presentation has the advantage that we do not have to say that the entries of $M$ are zero off $\mathfrak m$, as in \cite {BS} and \cite {JM}.  Again we may replaced a selected subset of the co-ordinates by $1$.  This means that we are evaluating $M^t$ on a linear subvariety of $\mathfrak m$.  Finally let $\textbf{1}$ denote the identity matrix, viewed as an element of $\End \mathbb C^n$.  Then we evaluate $M^t$ on $\textbf{1}+\mathfrak m$. This ``identity translate'' was motivated by Quantum Groups in which the Cartan subalgebra $\mathfrak h$ is replaced by a torus and the zero element in $\mathfrak h$ by the identity. The image of $M^t$ is not homogeneous and we define its top non-vanishing term to be $\gr M_{v,v'}$.  The following result appears in \cite {BS}. (It is easily checked).  Set $d_{v,v'}:=\deg \gr M_{v,v'}$.

\begin {lemma}  $$d_{v,v'}=\sum_{i=v+1}^{v'}\min \{s,n_i\}, \eqno {(3)}.$$
\end {lemma}

\subsubsection {} \label {4.1.7}

For our purposes an easy but key observation is that $m_{v,v'}-d_{v,v'}$ is exactly the number of boxes lying on the rows $R_i:i>s$ lying between the columns $C_v,C_{v'}$ (of height $s$).  Combined with Lemma \ref {4.1.5}, this gives the following result.

Two boxes $b_{u,v},b_{u',v'}$ in $\mathscr T$ are said to be strictly ordered if $v<v'$ and we write $b<b'$.

\begin {lemma}  The above construction gives a bijection between the standard matrix units $x_{i,j}\in \mathfrak m$ and pairs of strictly ordered boxes in $\mathscr T$ lying between the given neighbouring columns of height $s$ in the rows $R_i:i=1,2,\ldots,s$.
\end {lemma}


\

\subsection {The Construction of a Weierstrass Section} \label{4.2}

\

We represent the proposed Weierstrass section $e+V$ as follows.

Following Lemma \ref {4.1.7}, for each matrix unit $x_{i,j} \in \mathfrak m$, we draw a line joining the corresponding boxes in $\mathscr T$. We label every such line with a  $1$ or a $0$.

We will write $e$ as a sum of the form $x_{i,j}$ given by the lines labelled by a $1$ and $V$ as a direct sum of one dimensional subspaces $\mathbb Cx_{i,j}$ given by the lines labelled by $0$.

The way to specify which $x_{i,j}$ is labelled by $1$ or by $0$ is described below following three steps.

 Two distinct boxes in $\mathscr D$ are said to adjacent at level $u$ if they lie on row $R_u$ and there are no boxes in $\mathscr D$  in $R_u$ between them, equivalently that the columns between those containing the given boxes have height strictly less than $u$.

 Throughout all lines will be considered to be directed and going strictly from left to right.  For a particular line this has no meaning (except when we speak of its weight \ref {4.2.4}) but does have a meaning when lines are concatenated and in this the direction is preserved.  (No left-right zig-zagging allowed!)

\subsubsection {First Step}\label {4.2.1}

\

In the first step all horizontal lines between adjacent boxes on the same level are drawn.  This will mean that each box has at most one left (resp. right) going line.  This was the prescription used by Ringel et al \cite {BHRR} to construct an explicit element in $\mathfrak m$ generating a dense $P$ orbit.

\begin {lemma}  Set $s=\max_{i \in [1,r]}n_i$.  The number $m$  of horizontal lines is  $(\sum_{i=1}^r n_i)-s$.
In particular it is independent of the ordering of the set $\{n_i\}_{i=1}^r$.
\end {lemma}

\begin {proof} The assertion is immediate if there is just one column. Decreasing the length of the shortest column by $1$ decreases $m$ by $1$ and the above sum by $1$. This eventually reduces the number of columns by one and then the assertion follows by induction on the number of columns.
\end {proof}

\subsubsection {Second Step}\label {4.2.2}

\

In the second step every horizontal line is given the label $1$ with the following exceptions when it is given the label $0$.  Such an exception occurs exactly when there are two neighbouring columns of height $s$, say $C_v,C_{v'};v<v'$. Consider the horizontal lines joining adjacent blocks lying between $C_v,C_{v'}$.

Notice that the total number of zeros is just the number of pairs of neighbouring columns.

Choose $v''$ maximal such that $v\leq v''< v'$ and $C_{v''}$ has height $\geq s$.  Then the horizontal line joining the boxes $b_{s,v},b_{s,v''}$ is given the label $0$.  We call this a rightmost labelling by $0$. A leftmost labelling by $0$ is defined by requiring in the above that $v''$ be minimal such that $v< v''\leq v'$ and $C_{v''}$ has height $\geq s$.

It turns out that in order to carry out the induction procedure proving the separation result described in Proposition \ref {4.3.4}, we must consider both leftmost and rightmost labellings for all $s$.

There are two easy consequences of this construction. Set $s+1 = ht \mathscr D$.  Let $C_{v_1},C_{v_2},\ldots, C_{v_m}$ denote the columns of height $s+1$ in $\mathscr D$. Then

\

$(i)$. All the horizontal lines lying in $R_{s+1}$ are all labelled with $0$.

\

$(ii)$.   For a particular labelling, $b_{s,i}$ has at most one outgoing line with label $1$ for either $i=1$ or $i=m$.  Indeed if this were to fail, there would be a column $C_\ell$ (resp. $C_r$) of height $s$ to the left (resp. right)  of $C_{v_1}$ (resp. $C_{v_t}$) with no columns of height $\geq s$ in between. Then the assertion for $i=1$ (resp. $i=m$) results from a leftmost (resp. rightmost) labelling at level $s$.


 \subsubsection {Composite lines}\label {4.2.3}

A line is said to be a composite line if it is concatenation of lines joining boxes going strictly from left to right.

Two composite lines are said to be disjoint if they do not pass through a common box.

Notice that in the second step that between two neighbouring columns $C,C'$ of height $s$, there is for all $i=1,2,\ldots$,  a a unique composite line joining the box in $C\cap R_i$ to the box $C'\cap R_i$.  These lines are disjoint and their union meets every box in rows $\{R_i\}_{i=1}^s$ between $C,C'$.

\subsubsection {Moving lines and the third step}\label {4.2.4}

Consider the labelling of the horizontal lines given in \ref {4.2.2}. If there are neighbouring columns which overlap, then evaluation need not result in linear functions.  Thus the procedure must be modified in a third step.

This modification introduces lines which are not necessarily horizontal and do not necessarily join only adjacent boxes.  Moreover there may be more than one left (resp. right) going line from a given box. However all lines still go strictly from left to right; but they may ``hop over'' a box or they may cross, so we do not present them in figures as straight lines.

Our construction will be required to preserve the following property.

\

$(P_1)$.  For all $s \in \mathbb N^+$ and between any two neighbouring columns $C,C'$ of the same height $s$, there is permutation $\sigma$ of $1,2,\ldots,s$ and  a unique disjoint union of $s$ composite lines joining $C\cap R_i$ to $C'\cap R_{\sigma (i)}$.

\

However our construction will be required to also satisfy the additional property.

\

$(P_2)$.  Of the lines which make up the composite lines in a disjoint union joining the boxes of $C,C'$ given by $(P_1)$, all are labelled by $1$ except for exactly one which is labelled by $0$.

\

Step $3$ requires deleting some lines and adding others, but not too many!

The condition that there is a unique union of disjoint lines in $(P_1)$ has the immediate consequence that there is at most one line joining $b:=b_{u,v}$ to $b':=b_{u',v'}$ and only when $v<v'$ and in this case set $b<b'$.  We denote this line if it exists by $\ell_{b,b'}$.

The weight attached to $\ell_{b,b'}$ is the weight of the corresponding element of $\mathfrak m$ and denoted by  $\varpi_{b,b'}$.

Let $i$ (resp. $j$) be the entry in $b_{u,v}$ (resp. $b_{u',v'}$).  Then  $\varpi_{b,b'}= \alpha_i+\ldots+\alpha_{j-1}$.

In this we may replace the boxes by their entries.

The union of the lines $\ell_{i,j},\ell_{j,k}$ with $i,j,k$ increasing is a composite line denoted by $\ell_{i,j,k}$ (and so on).

 \

\textbf{Example}.  Let $\mathscr D$ be defined by the array $2,1,1,2$.  It admits two sets of neighbouring columns.  In the first step we draw the horizontal lines $\ell_{1,3},\ell_{3,4},\ell_{4,5},\ell_{2,6}$.  In the second step, we attach $1$ to the first and third lines and $0$ to the second and fourth lines.  The evaluation of the $2\times 2$ minor corresponding to the inner pair of neighbouring columns is just $x_{3,4}$.  However the second minor evaluates to $x_{3,4}x_{2,6}$.

To correct for this we make a ``rectangular'' modification. This entails removing $\ell_{2,6}$ and introducing the lines $\ell_{2,4},\ell_{3,6}$ with the former carrying $1$ and the latter $0$.  The evaluation of the $2 \times 2$ minor is unchanged, whilst the second minor evaluates to $x_{3,6}$.  Notice that the only disjoint union of composite lines passing through all the boxes between the outer neighbouring columns are $\ell_{1,3,6},\ell_{2,4,5}$.  Step three will be a careful generalization of this process.

\begin{figure}[H]
\begin{tikzcd}[row sep=tiny, 
column sep = 1 em]
1\arrow[-,r]&3\arrow[-,r]&4\arrow[-,r]&5\\
2\arrow[-,rrr]&&&6\\
&&\arrow[-,"{\text{ step 1}}"]&\\
\end{tikzcd}
\hspace{1 em}
\begin{tikzcd}[row sep=tiny, 
column sep = 1 em]
1\arrow[-,r,"1"]&3\arrow[-,r,"0"]&4\arrow[-,r,"1"]&5\\
2\arrow[-,rrr,"0"]&&&6\\
&&\arrow[-,"{\text{ step 2}}"]&\\
\end{tikzcd}
 \hspace{1 em}
 \begin{tikzcd}[row sep=tiny, 
column sep = 1 em]
1\arrow[-,r,"1"]&3\arrow[-,rrd,"0"]&4\arrow[-,r,"1"]&5\\
2\arrow[-,rru,"1"]&&&6\\
&&\arrow[-,"{\text{ step 3}}"]&\\
\end{tikzcd}
\caption*{The rectangular modification in step 3}
\end{figure}
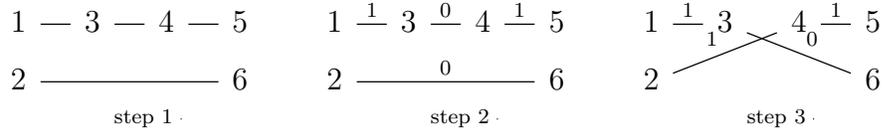

\subsubsection {Evaluation of $M^t$}\label {4.2.5}

Let $C_v,C_{v'};v<v'$ be neighbouring columns of height $s$.  We should like the evaluation of $M^t_{v,v'}$ on the linear subvariety $e+V$ to be the co-ordinate function on $x_{i,j}$ (namely $x^t_{i,j}:=x_{j,i}$) defined by the line with label $0$.

This meets with a further difficulty.  It arises when the expression $m_{v,v'}-d_{v,v'}$ introduced in \ref {4.1.7} is non-zero.  This is easily overcome as follows.  Let $L$ denote the set of entries in the boxes in the rows $R_u:u>s$ and lying between the columns $C_v,C_{v'}$. Observe that $$|L|=m_{v,v'}-d_{v,v'}. \eqno {(4)}$$

Set $x_{\ell,\ell}=1$ for all $\ell \in L$. This it where the ``identity translate'' motivated by Quantum Groups - \ref {4.1.6} - is needed.

We now assume property $(P_1)$.

Consider the disjoint union of composite lines specified by $(P_1)$. It follows from Lemma \ref {4.1.5} that the co-ordinate functions $x^t_{i,j}$ defined by the individual lines, lie in $\mathfrak m$ and are entries of $M^t_{v,v'}$.  Define the  multi-set $I'$ (resp. $J'$) given by the first co-ordinate to be the set of the starting (resp. finishing) point of the set of individual lines.  The non-intersection property means that these are actually sets, (that is  multiplicity-free), though the intersection $I'\cap J'$ can be non-empty.

The set of lines joining $i \in I'$ to $j \in J'$ defines a bijection $\varphi$ of $I'$ onto $J'$. Thus $|I'|=|J'|$ and this common cardinality is obviously the total number of boxes in $\mathscr D$ between $C,C'$ strictly above row $R_{s+1}$ minus $s$, that is
$$|I'|=|J'|=\sum_{i=v}^{v'}\min \{s,n_i\}-s=d_{v,v'}. \eqno {(5)}$$

Let $\varphi$ denote the restriction of $\mathbb C[\mathfrak m]^{\mathfrak p'}$ to $\textbf{1}+e+V$.

\begin {lemma}  Assume $(P_1)$ and let $C_v,C_{v'}$ be neighbouring columns of height $s$.

\

(i)  The evaluation of $M^t_{v,v'}$ on $\textbf{1}+e+V$ is the product of the co-ordinate functions on the lines labelled by $0$, lying between $C,C'$.

\

(ii) The image $\varphi$ has the same field of fractions as $\mathbb C[V]$.

\

(iii)  Suppose $(P_2)$ also holds, then the image of $\varphi$ is $\mathbb C[V]$.

\end {lemma}

\begin {proof}

\

We claim that
$$I'=I\setminus \{L\},J'=J \setminus \{L\}. \eqno {(6)}$$

Indeed by Lemma \ref {4.1.5}, one has   $I' \subset I,J'\subset J$. Again $L$ lies in the diagonal entries of each of the blocks given by the columns $C_j: v<j<v'$ and therefore $x_{j,j}$ is an entry of $M_{v,v'}$.  On the other hand, by definition of $I',J'$, these index sets come from boxes distinct from those in rows $R_i:i>s$.  Thus it follows that $K$ has null intersection both with $I'$ and with $J'$.  Then combining equations $(2-5)$ we obtain $(6)$.

The monomial obtained from the (standard) development of the minor given by $\prod_{k \in L} x^t_{\ell,\ell}\prod_{i\in I'}x^t_{i,\varphi(i)}$ evaluates to the product of the co-ordinate functions on the lines to which $0$ is attached.

The uniqueness of the union of composite lines implies that there can be no other monomial in the (standard) development of the minor which has a non-zero evaluation.  Hence (i).

The lines carrying a $0$  occur on rows $R_i:i \leq s$ and between $C_v,C_{v'}$. Thus (ii) obtains by induction on height. When $(P_2)$ holds a minor restricts to a single co-ordinate function on $V$ and all  co-ordinate functions on $V$ obtain (by definition of $V$).   Hence (iii).

\end {proof}

\textbf{Remark}. By Lemma \ref {4.1.7}, the degree of the product of the matrix units coming from the lines between $C,C'$ is the degree of $\gr M_{v,v'}$.

\subsection {Separation}\label {4.3}

Admit the result (basically of Melnikov) in \cite {JM} that the number of generators of the polynomial algebra $S(\mathfrak m^*)^{\mathfrak p'}$ is the number of pairs of neighbouring columns in $\mathscr D$. Then by (ii) of Lemma \ref {4.2.5} it follows that $\varphi$ is injective.  However this result of Melnikov involved identifying and then counting hypersurface orbital varieties which is a tricky business especially outside type $A$.  Thus we might wish to give an alternative proof. For this it is enough to show that $P'(e+V)$ is dense in $\mathfrak m$.  In this part of the work has already been done for us by Ringel et al \cite {BHRR}.  Indeed let $\hat{e}$ be defined by placing $1$ on every horizontal line. Then \cite {BHRR} asserts that $P\hat{e}$ is dense in $\mathfrak m$ giving thereby an immensely complicated proof of Richardson's theorem; but one which nevertheless gives some help in constructing a Weierstrass section.  Moreover in \cite {B}, Baur has shown how to generalize their result to all classical Lie algebras.

The first step towards our goal is to prove the separation theorem below.

\subsubsection {}\label{4.3.1}

Return to the situation described by the second step.  Let us compute the number of horizontal lines carrying a $1$.  Since the number carrying a $0$ is just the number of pairs of neighbouring columns of the same height and so independent of permutation, we can assume by Lemma \ref {4.2.1}, that the $n_i:i=1,2,\dots,r$ are decreasing.  Thus we can write $\{n_i\}_{i=1}^r$ as $\{m_i^{k_i}\}_{i=1}^s$, where $m_i^{k_i}$ denotes $k_i$ copies of $m_i \in \mathbb N^+$, with the latter assumed strictly decreasing.  Taking $m_0=0$, we obtain $m_s-s=\sum_{i=1}^s(m_i-m_{i-1}-1)$.  Thus $m_s-s$ is the number of ``gaps'' in the sequence $m_1,m_2,\ldots,m_s$ and is a non-negative integer.

Set $\mathfrak h'=\mathfrak h\cap [\mathfrak r, \mathfrak r]$.  Its dimension is the cardinality of $\pi'$.

\begin {lemma}  The number of horizontal lines carrying a $1$ is $\dim \mathfrak h'-(m_s-s)\leq \dim \mathfrak h'$.
\end {lemma}

\begin {proof}  Since the $n_i$ can be assumed decreasing, it follows easily that the number of horizontal lines carrying a $1$ is just $\sum_{i=1}^s(k_i-1)(m_i-1) +\sum_{i=1}^{s-1}m_i$, whilst the dimension of $\mathfrak h'$ is $\sum_{i=1}^sk_i(m_i-1)$.  Combined these prove the required assertion.
\end {proof}

\subsubsection {}\label{4.3.2}
%

Let $K$ be the set of horizontal lines carrying a $1$, as described through the construction of \ref {4.2.2}.  For all $k \in K$, let $\varpi_k$ denote the weight of the vector in $\mathfrak m$ defined by this line (cf \ref {4.2.4}).

 We would like to show that $\mathfrak h'$ separates the $\{\varpi_k\}_{k \in K}$.  In general such questions, though only elementary linear algebra can be notoriously difficult.

 In part of our analysis we do not need the lines to be horizontal.  Ultimately we shall use the conclusions of (i) and (ii) of \ref {4.2.2} though no doubt a more general result is possible.  Yet necessarily the number of lines (carrying $1$) which can be separated must not exceed $\dim \mathfrak h'$. This will fail at the third step.  Thus whatever happens only some lines carrying a $1$ can be included and it is rather convenient to just assume these are all horizontal and given by the construction of \ref {4.2.2}.

  \subsubsection {}\label{4.3.3}

     Choose $ s \in \mathbb N^+$ such that $ht \mathscr D =s+1$. We can assume that $s>0$ since for $s=0$ all lines are horizontal and all carry a zero.   This is the case when the parabolic is the Borel.

   Let $\mathscr T'$ denote the tableau obtained from $\mathscr T$ by eliminating the lowest box on every column and let $K'$ denote its set of entries.  One easily checks that $\mathfrak h'$ is just the linear span of $\{\alpha_{k'}^\vee\}_{k' \in K'}$.

   Fix a pair of strictly ordered boxes $b,b'$ in $\mathscr T$ not necessarily on the same row.

   \begin {lemma} Let $C$ be a column of $\mathscr T$ and $k'$ an entry of $C\cap \mathscr T'$ in row $R_t$.  Then

   \

   (i) $\alpha_{k'}^\vee(\varpi_{b,b'})=0$, unless $b$ or $b'$ belong to $C$.

   \

   (ii) $\alpha_{k'}^\vee(\varpi_{b,b'})=0$, unless $b$ or $b'$ belongs to $C\cap R_{t'}$, with $t'\geq t$.
   \end {lemma}

   \begin {proof}  (i). Let $k_m$ (resp. $k_r$) be an entry of a column strictly to the left (resp. right) of $C$.  Then $k_m<k'<k_r-1$ (because $k'+1$ is an entry of $\mathscr D \cap C$).  Thus if $b$ is in a column strictly to the left of $C$ and $b'$ is in a column strictly to the right of $C$, then  $\varpi_{b,b'}=\alpha_{k_m}+\ldots + \alpha_{k_r-1}$ vanishes on $\alpha_{k'}^\vee$.

    Again if $b'$ (resp. $b$) is in a column strictly to the left (resp. right) of $C$, then $\varpi_{b,b'}$ ends (resp. begins) at $\alpha_{k_{m}-1}$ (resp. $\alpha_{k_r}$) and again vanishes on  $\alpha_{k'}^\vee$.

    For (ii), suppose $b' \in C$ and has entry $m'$.  By convention $b$ lies strictly to the left of $b'$.  Then $\varpi_{b,b'}$ ends $\alpha_{m'-1}$.  If $t'<t$, then $m'<k'$ and so the assertion results.  Finally suppose  $b \in C$ and has entry $m'$.  By convention $b'$ lies strictly to the right of $b$.  Then $\varpi_{b,b'}$ begins in  $\alpha_{m'}$ and as in the first part ends in some $\alpha_{m-1}$ with $m>k'+1$.  If $t'<t$, then $m'<k'$ and so the assertion follows.
   \end {proof}

 \subsubsection {}\label{4.3.4}

 Recall the notation and hypotheses of \ref {4.2.2}, in particular the labelling of the horizontal lines.

 Recall the definition of $K$ (resp. $K'$) given in \ref {4.3.2} (resp. \ref {4.3.3}).

 Through the above lemma and (ii) of \ref {4.2.2} we obtain the

 \begin {cor}   Set $ht \mathscr D =s+1$.

 \

 (i)  Suppose that there is just one column $C'$ of height $s+1$ and every other column has height $s'<s$. Let $k'$ be the entry of $R_s\cap C'$.  Then $\alpha^\vee_{k'}$ vanishes on $\varpi_k$, for all $k \in K$.

 \

 (ii)  Suppose that the number of columns of height $\geq s$ is strictly greater than one. Let $C'$ (resp. $C''$) be the leftmost (resp. rightmost) column of height $s+1$.  Then either $C'$ or $C''$ admits a box $b \in R_s$ with just one outgoing line $\ell_b$ with label $1$ and let $k'$ be the entry of that box.  Then $\alpha_{k'}^\vee$ is non-zero on exactly one element of $\{\varpi_k\}_{k \in K}$, namely that defined by $\ell_b$.

 \end {cor}

 \subsubsection {}\label{4.3.5}

Continue to take $ht \mathscr D=s+1$. We now show that $\{\varpi_k\}_{k \in K}$ is separated by $\mathfrak h'$.   More precisely

\begin {prop}  The $\{\varpi_k\}_{k \in K}$ restricted to $\mathfrak h'$ are linearly independent.
\end {prop}

\begin {proof} The proof is by induction of the number of boxes in $\mathscr T$.

Set $\ell'=\dim \mathfrak h'= |K'|$ and $\ell=|K|$.

We need to show that the $\ell \times \ell'$ matrix with entries $\{\varpi_{k}(\alpha^\vee_{k'})\}_{k\in K,k'\in K'}$ has rank $\ell$.

Recall the conclusion of the corollary. Let $C$ be either $C'$ or $C''$ in its conclusion admitting at most one outgoing line in $R_s$ with label $1$.  let $k'$ denote the entry of the box $b:=C\cap R_s$.  If there is no outgoing line with label $1$ from $b$, then hypothesis of (i) holds and we define $K_C:=\phi$.   Otherwise define $K_C:=\{k\}$, where $k\in K$ is the line defined by $\ell_b$.  In the latter case $\alpha^\vee_{k'}(\varpi_k)\neq 0$.

\

$(*)$.  Let $\hat{b}$ be the box containing $k'+1$.  Eliminate $\hat{b}$ and reduce by $1$ the label of every box having entry $>k'+1$.

\

This gives a new tableau $\mathscr T_-$. It has one less box than $\mathscr T$ and so we can assume that the conclusion of the proposition applies to it.

Identify the set $\hat{K}$ of \textit{all} lines in $\mathscr T$ not outgoing from the box $\hat{b}$, with those of $\mathscr T_-$.

Observe that $\hat{K}\supset K$, since the lines outgoing from $\hat{b}$ all carry a $0$.  Of course $\hat{K} \setminus K$ is the set of the remaining lines carrying a $0$.

It is convenient to consider the larger matrix with rows defined by $\hat{K}$.

\

\textbf{Claim.}  The larger matrix  defined by $\mathscr T_-$ with rows defined by $\hat{K}$ is that obtained by eliminating the column containing the label $k'$ from the larger matrix defined by $\mathscr T$ and rows defined by $\hat{K}$.

\

Let $C_1$ be a column  different from $C$.  Let $m'$ be the entry of a box $b$ in $C_1$ and consider the matrix elements obtained by evaluation at $\alpha^\vee_{m'}$.

Observe that if $C_1$ is to the left (resp. right) of $C$, then $m'<k'$ (resp. $m' >k'+1$).

Suppose $C_1$ is to the left of $C$ and $\ell_b$ an outgoing line from a box $b$ in $C_1$. Since $C$ has the maximal height, $\ell_b$ must end to the left of $C$, or at $C$. Thus its weight $\varpi_{\ell_b}$ must end in some $m-1$ with $m \leq k'+1$.

Suppose $C_1$ lies to the right of $C$.  Let $b$ be a box of $C_1$ and $\ell_b$ an outgoing line from $b$. Let $m$ be the entry of $b$, then $m>k'+1$. Let $\varpi_{\ell_b}$ be the weight of $\ell_b$.  If $\ell_b$ is right going then $\varpi_{\ell_b}$ must begin at $\alpha_m$. If $\ell_b$ is left going and does not end in $C$, then  $\varpi_{\ell_b}$ must begin at some $\alpha_{m'}$ with $m' >k'+1$.

Thus the procedure defined by $(*)$ does not change the above matrix elements, since (fortuitously!) scalar products in type $A$ do not change on translation of the Dynkin diagram.

Taking account of Lemma \ref {4.3.3}, we conclude that the only matrix elements that can possibly change are the $\alpha_{m'}^\vee(\varpi_{\ell_b})$, where $m'$ is the label of a box in $C$ and $\ell_b$ is a right going line from a box in $C$.
In this $m'<k'$, otherwise these matrix elements are eliminated by $(*)$.

One can choose positive integers $m_1\leq m_2$ such that $\varpi_{\ell_b}=\sum_{i=m_1}^{m_2}\alpha_i$.  Since $k'+1$ is an entry of $C$, one has $m_2\geq k'+1$, whilst by Lemma \ref {4.3.3}(ii), we can assume $k' \geq m_1 \geq m'$, otherwise $\alpha_{m'}^\vee(\varpi_{\ell_b})=0$.

 Now $(*)$, leaves $m'$ fixed.  Again if $i <k'$ (resp. $i=k',i>k'$) then in the sum $(*)$ leaves $i$ fixed (resp. eliminates $i$, reduces $i$ by $1$).  Thus the new weight becomes $\varpi_{\ell_b}=\sum_{i=m_1}^{m_2-1}\alpha_i$.  Consequently the matrix element defined by $\mathscr T$ minus that defined by $\mathscr T_-$, differs by  $\alpha^\vee_{m'}(\alpha_{m_2})$, which (fortuitously!) equals zero.

 This proves the claim.

 \

Let us now show how the proposition results from the claim.

We first return to the matrices defined by $K$.  Then we show that the matrix obtained from $\mathscr T$ in which the column defined by $k'$ and the row defined by $K_C$ are eliminated, has rank at least $\ell-|K_C|$.

In $K_C$ is empty, then by the induction hypothesis the matrix defined by $\mathscr T_-$ has rank $\ell$ and by the claim so has the matrix defined by $\mathscr T$.

We are left with the case $K_C = \{k\}$. By definition $k$ is the unique outgoing line carrying a $1$ from $C$ to a column $C_1$ of height $\geq s$.

Recall that $ht \mathscr T = s+1$. Suppose there is no column of height $s$ in $\mathscr T$.  Then the sequence that defines $\mathscr T$ has a gap in the sense of \ref {4.3.1} and so $\ell'>\ell$, by Lemma \ref {4.3.1}.  In the matrix defined by$\mathscr T_-$, the dimension of $\mathfrak h'$, that is to say $\ell'$ is reduced by one, but the number of lines remains unchanged, so $\ell$ is also unchanged. By the induction hypothesis the matrix defined by $\mathscr T_-$  has rank $\ell$, so when the row defined by $k$ is removed, it comes to have rank $\ell-1$.  By the claim this is the rank of the matrix defined by $\mathscr T$ when the column defined by $k'$ and the row defined by $k$ are removed.  Recalling the column defined by $k'$ has only one non-zero entry and this occurs on the row defined by $k$, it follows that the matrix defined by $\mathscr T$ has rank $\ell$.

Finally suppose there is a column $C_2$ in $\mathscr T$ of height $s$.  Recall that $C$ has height $s+1$. Then in $\mathscr T_-$ there are at least two columns of height $s$, namely $C_2$ and the column obtained from $C$ by deleting $\hat{b}$.  Then $0$ can be placed on the line defined by $k$ in $\mathscr T_-$ using as necessary a left or right most labelling (\ref {4.2.2}).  Then this line will not be counted in the matrix defined by $\mathscr T_-$.  Yet by the induction hypothesis this matrix has rank $\ell-1$ and so, as in the previous paragraph,  the matrix defined by $\mathscr T$ has rank $\ell$.

This completes the proof of the proposition.

\end {proof}

 \subsubsection {}\label{4.3.6}

 \textbf{Example 1.}  This exemplifies the very last step of the above proof.

 Take $\mathscr T$ to be given by the array $1,3,3,2$.  Since $C_3\cap R_2$ has right going line labelled by $1$, we are forced to take $C'=C_2$ to apply Corollary \ref {4.3.4}(ii) and then $k'=3$ being the entry of $C'\cap R_2$.  Now $\alpha_{k'}^\vee$ is only non-vanishing on the weights defined by the lines $\ell_{3,6}$ and $\ell_{4,7}$.  The latter carries a zero, so does not contribute to the matrix.  Applying $(*)$, results in $\mathscr T_-$ being given by the array $1,2,3,2$.  Passing to this new array the line $\ell_{3,6}$ (which is unique $k \in K$ for which $\alpha_{k'}^\vee(\varpi_k)\neq 0$ and so is eliminated in computing the minor) is replaced by the line $\ell_{3,5}$.   Thus the latter also needs to be eliminated.  This is possible as there is a pair of neighbouring columns of height $2$, and we take a leftmost labelling by $0$ in $\mathscr T_-$.  By contrast for the array $2,3,3,1$ we must take a rightmost labelling.

  \textbf{Example 2.} This exemplifies the penultimate step of the above proof.

  Take $\mathscr T$ to be given by the array $3,3,1$.  It has a gap. Take $C=C_2$.  It has a left going line $\ell_{2,5}$ from $b_{2,2}$ which has entry $5$. Applying $(*)$ of \ref {4.3.5}  results in $\mathscr T_-$  being given by the array $3,2,1$.  It has no gaps and so the rank of the matrix it defines is $3$.  Then we can simply eliminate the line corresponding to $\ell_{2,5}$.  This drops the rank by $1$.  Thus we conclude that the matrix defined by $\mathscr T$ has rank $3$ as required.

  \

  \textbf{Remark.}  Inspection of our reduction shows that there is always an $\ell \times \ell$ minor in the matrix defined by $\mathscr T$ which is triangular up to permutations of rows and columns, whose diagonal entries are $\pm 1$.

\subsection{Performing Step $3$}\label{4.4}

\

Recall that after the first step (\ref {4.2.1})
all the horizontal lines on adjacent boxes in a given row are joined.

In the second step the lines are labelled by either $1$ or $0$ as described in \ref {4.2.2} and here we shall use the rightmost labelling by $0$. This means that for each pair of neighbouring columns $C,C'$ of height $i$, the rightmost line on the composite line joining $C\cap R_i$ to $C'\cap R_i$ is labelled by $0$, all other lines on this composite line are labelled by $1$.

Any two boxes on the same row are joined by a composite line, thus property $(P_1)$ holds, but property $(P_2)$ may fail.

To obtain Property $(P_2)$ without losing property $(P_1)$ is in principle quite easy.  One selects a pair of neighbouring columns $C,C'$ of height $i$. One then ``gates'' lines with a $0$ on previous rows lying between $C,C'$ and removes some specific lines on $R_i$. Then the boxes losing a left or right going line either by gating or removal are rejoined to recover $(P_1),(P_2)$.  It turns out that gating, removing and rejoining is ``canonical'', so the reader can either guess the way that Step $3$ is established or read the details below.

\

A gated or ungated line will always refer to a line carrying a $0$.

\

Given two boxes $b,b'$, not necessarily in the same row, we write $b<b'$ (resp. $b\leq b'$) if $b$ lies in a column strictly to the left (resp. to the left) of $b'$. Note that equality does not mean that the boxes are the same, merely that they are on the same column.  the strictly inequality means that $b$ can be connected to $b'$ by a line which may form part of a composite line going strictly monotonously from left to right.

\subsubsection{}\label{4.4.1}

For all $i \in \mathbb N^+$, let $R^i$ denote the union of the rows $\{R_j\}_{j=1}^i$.

The modification to obtain $(P_2)$ is made inductively down the rows. At the $i^{th}$ (induction) stage lines carrying a $0$ in $R^{i-1}$ are gated, some lines in $R_i$ are removed and then new lines are drawn and labelled.  These new lines need \textit{not} be horizontal and may slide between columns.

Call a box left (resp. right) extremal if it has no left (resp. right) going lines.  The left (resp. right) extremal boxes at Step $2$ are exactly those which have no adjacent box to the left (resp. right) in $\mathscr D$.

Our construction does not add a left (resp. right) going line to a box which has no left (resp. right) going line.  Thus the extremal boxes remain the same.


A line can consist of a single box $b$.  This occurs exactly when $b$ is isolated in its row.

\subsubsection{}\label{4.4.2}

At the $(i-1)^{th}$ stage, we assume as a first part of the induction hypothesis, that the following modification of $(P_1)$ holds.

\

$(P_1^{i-1})$.  The boxes in $R^{i-1}$ lie on a unique disjoint union of $(i-1)$ composite lines.


\

%

If there are no ungated lines carrying a $0$ in $R^{i-1}$ between neighbouring columns of height $i$, then $R_i$ (which consists of a single composite line) is adjoined to $R^{i-1}$.

\subsubsection{}\label{4.4.3}

The second part of the induction hypothesis is that the lines in $R^{i-1}$ that are \textit{not} gated (so by convention carry a $0$) and lie between the neighbouring columns of height $i$, join $u$ pairs of boxes $b_j,b_j':j=1,2,\ldots,u:b_j, b_j' \in R^{i-1}$ with $b_j<b_j'\leq b_{j+1}$.

Only the last part of this property is non-trivial. It will be called the strict non over-lapping property of the ungated lines at the $(i-1)^{th}$ step.

\subsubsection{}\label{4.4.4}
 Fix a pair of neighbouring columns $C,C'$ of height $i$.

 \

 $(*)$.  Consider only the ungated lines in $R^{i-1}$ which lie \textit{between} $C,C'$.

 \

To simplify notation we shall take the labelling of the pairs of boxes $b_j,b_j'$ given in \ref {4.4.3} to just refer to the lines between $C,C'$.  We may assume $u>0$, otherwise there are no such lines.

Gate the lines between the pairs $b_j,b_j':j=1,2,\ldots,u:b_j, b_j' \in R^{i-1}$.  (These all carry a zero.)

Recall that the composite line (which lies entirely in $R_i$), joining $C\cap R_i$ to $C'\cap R_i$ consists of horizontal lines between adjacent boxes $b''_1<b''_2<\ldots< b''_v$. It has just one horizontal line with label $0$.  Since we are taking a rightmost labelling of $0$, this line joins $b''_{v-1}$ to $b''_v$.

By $(*)$, it follows that $b_v'' \geq b'_u$, hence $b_v''> b_u$ and  $b_1'' \leq b_1$, hence $b_1''<b_1'$.

\subsubsection{}\label{4.4.5}

 Choose $v_1$ maximal such that $b_{v_1}'' < b_1'$.  This forces $v_1<v$.  If $b_1\geq b''_{v_1+1}$, then $b_1'>b''_{v_1+1}$ contradicting the choice of $v_1$ so we also have $b_1<b''_{v_1+1}$.

 Delete the lines joining $b_{v_1}'',b_{v_1+1}''$ and $b_{v-1}'',b_v''$.  Since we are using a rightmost labelling, this is just one line if $v_1=v-1$ and carries a $0$, otherwise it is two lines the former carrying a $1$, the latter a $0$.

 \subsubsection{}\label{4.4.6}

 Join the boxes $b_{v_1}'',b_1'$ with a line carrying a $1$ and the boxes $b_u,b_v''$ with a line carrying a $0$. The latter is the unique new ungated line at the $i^{th}$ step between $C,C'$.

 \subsubsection{}\label{4.4.7}

 By the strict non over-lapping of ungated lines at the $(i-1)^{th}$ step, as defined in \ref {4.4.3}, we obtain a sequence of boxes and inequalities
 $$b_1<b_1'\leq b_2<b_2'\leq \ldots b_{u-1}<b'_{u-1}\leq b_u <b'_u.$$

  \subsubsection{}\label{4.4.8}

   By \ref {4.4.6}, $b_1'$ (resp. $b_u$) has a left (resp. right) going line.

  Suppose $v_1=v-1$.  Then $b_i:1\leq i <u$ (resp. $b'_i:1<i\leq u$) has no right (resp. left) going line. Then join $b_j$ to $b'_{j+1}$ for all $j \in [1,u-1]$.  This operation is trivial if $u=1$.

  \subsubsection{}\label{4.4.9}

  The case $v_1<v-1$ described below may seem more complicated but follows the same principle.  Indeed If $b,b'$ are joined by a line in $R^{i-1}$ carrying a $0$, gated at the $i^{th}$ step or a line in $R_i$ which is removed, then our construction has as its goal to give a right going line from $b$ and a left going line from $b'$. (The reader may check that there is only one way in which this can be done.)

  The complication is caused by columns of height $>i$ between $C,C'$, whilst we recall that the latter have height $i$. We shall label these columns in a particular fashion.  This can be achieved in one operation but it is less confusing when presented in three operations.

   \subsubsection{}\label{4.4.10}

   \

  \textbf{Operation $1$.} First label the columns of height $>i$ between $C,C'$ by  $C_1,C_3,\ldots, C_m$ going from left to right.
  Set $C_0=C,C_{m+1}=C'$.

   \

   \textbf{Operation $2$.}  Take $k\geq 1$ maximal such that there are no ungated lines in $R^{i-1}$ between $C_0,C_k$ and set $C_0'=C_k,C_1=C_{k+1}$.
    In this either $C_0'=C_0$ or there are columns of height $>i$ between $C_0,C_0'$.

     By choice of $k$, there is at least one ungated line between $C_k,C_{k+1}$ and that $b''_{v_1}$ lies in $C_0'$.  Repeat this process with $C_1$ replacing $C_0$ and so on.  Forget the labelling of operation $1$!

   \

   \textbf{Operation $3$.}  Now again relabel the columns by setting $C_{2j}=C_j, C_{2j+1}=C'_j$, for all $j=0,1,\ldots$.  Forget the labelling of operation $2$!

   \

    This process has the following two results.

    \

   \textbf{$A$.}  For $j$ even, one can have $C_j=C_{j+1}$ or being distinct with possibly some columns of height $>i$ between them; but no ungated lines in $R^{i-1}$ between them.

   \

   \textbf{$B$.}  For $j$ odd, the columns $C_j,C_{j+1}$ are distinct with no columns of height $>i$ strictly between them; but at least one ungated lines in $R^{i-1}$ between them.

   \


   Note that in this notation $C'=C_{2n}$ for some positive integer $n$.

   Let $b''_{v_j}$ be the entry of $C_j\cap R_i$, for $j=1,2,\ldots,2n$.  By the remark in Step $2$, this is compatible with our previous definition of $b''_{v_1}$.  Again one checks that $v=v_{2n}$.



  Now take $j\in \{2,3,\ldots,2n-1\}$.

  For $j$ even, let $u_j$ be maximal such that $b_{u_j}<b''_{v_j}$.
  For $j$ odd, let $u'_j$ be minimal such that $b''_{v_j}<b'_{u'_j}$.

  Join $(b_{u_j},b''_{v_j})$, for $j$ even  and join $(b''_{v_j},b'_{u'_j})$, for $j$ odd, with a line carrying a $1$.  Since there are no ungated lines between $C_{2j},C_{2j+1}$ one has $u'_j=u_j+1$. Thus compared to \ref {4.4.8}, the pairs $(b_{u_j},b'_{u'_{j+1}})$, with $j$ even should be exactly the pairs in $R^{i-1}$ which are \textit{not} joined.

  Now take $j\in \{1,2,\ldots,2n-1\}$.  Then the lines in $R_i$ joining $b''_{v_j},b''_{v_{j+1}}$ with $j$ odd, must be deleted.  (For $j=1,2n-1$, this was already undertaken in \ref {4.4.5}.)
  Notice that if we set $b'_{u'_1}=b'_1,b_{u_{2n}}=b_u$, then the pairs
  $(b''_{v_1},b'_{u'_1}),(b_{u_{2n}},b''_{v_{2n}})$ were already joined in \ref {4.4.6}.

 This achieves the  goal announced in \ref {4.4.9}.

  \subsubsection{}\label{4.4.11}

\begin {prop} After the $i^{th}$ stage has been carried out $(P_1^i)$ holds.  Moreover the strict non-overlapping property of the ungated lines holds at the $i^{th}$ stage.
\end {prop}

\begin {proof}

If $b,b'$ are joined by a line in $R^{i-1}$ carrying a $0$, gated at the $i^{th}$ step or a line in $R_i$ which is removed, then (\ref {4.4.9}) that our construction gives a right going line from $b$ and a left going line from $b'$.

We conclude that after the $i^{th}$ step there is unique disjoint union of composite lines, whose union passes through every box in  $R^i$ and does not include any line which has been gated or (of course!) removed.  Moreover it takes the left extremal boxes to the right extremal boxes up to a permutation (which in general is different to the one obtained from $(P_1^{i-1})$).

Of course this is not quite the end of the story since gated lines are not actually removed, indeed they are needed for minors associated with neighbouring columns of height $<i$.

To complete the proof, let us suppose that there is a second disjoint union of composite lines whose union passes through every box in $R^i$. In this we can assume that there are just two columns of height $i$, as the general case is similar.  Moreover we can assume that the line in $R_i$ carrying a $0$ has been removed, otherwise the lines in $\mathscr T$ are unchanged and $(P_1^{i-1})$ gives $(P_1^i)$.

The \textit{complete!} removal of the line in $R_i$ joining $b''_{v_{2n-1}},b''_{v_{2n}}$ carrying a zero means that the original horizontal composite line in $R_i$ cannot be used in the second union.  Yet $b''_{v_{2n}}$ is not left extremal and is joined at the $i^{th}$ step (see \ref {4.4.6}) by a left going line meeting $b_u$. Consequently the gated line joining $b_u,b_u'$ cannot appear in a composite line of our second union. Again since the line joining $b''_{v_1},b''_{v_2}$ is \textit{completely!} removed, the gated line joining $b_1,b_2'$ cannot appear in a composite line of our second union.  Continuing on in this fashion we conclude that successively none of the gated lines joining $b_j,b_j'$, for $j\in [1,u-1]$ can appear in a composite line of our second union.

Thus uniqueness follows from uniqueness in the first part in which the use of gated lines was forbidden.

The last assertion of the proposition is immediate for the new ungated lines added at the $i^{th}$ stage.  Indeed they all lie in $R_i$ and between successive columns of height $i$.  On the other hand the remaining ungated lines lie outside the columns of height $i$ and were obtained before or at the $(i-1)^{th}$ stage, so satisfying the strict non-overlapping property at the $(i-1)^{th}$ stage by the induction hypothesis.
\end {proof}

  \subsubsection{}\label{4.4.12}

For neighbouring columns $C,C'$ of height $s$, we obtain $(P_1)$ from $(P_1^s)$.  Moreover $(P_2)$ obtains, since there is just one line carrying a zero in these disjoint composite lines, namely that joining $b_u,b''_v$, in the notation of \ref {4.4.6}.

This completes Step $3$.

The proof of the proposition shows that it may be amply illustrated when there are just columns of heights $1,2,3$ and when just the first and second steps are carried out.  An example is given in Figure $1$.  In Figure $2$ we give an example when both $R_1,R_2$ have gated lines at stage $3$.

\

\textbf{Example.} Let $\mathscr T$ be defined by the array $(2,3,1,1,1,3,3,,1,1,1,1,3,3,3,1,1,2)$ 
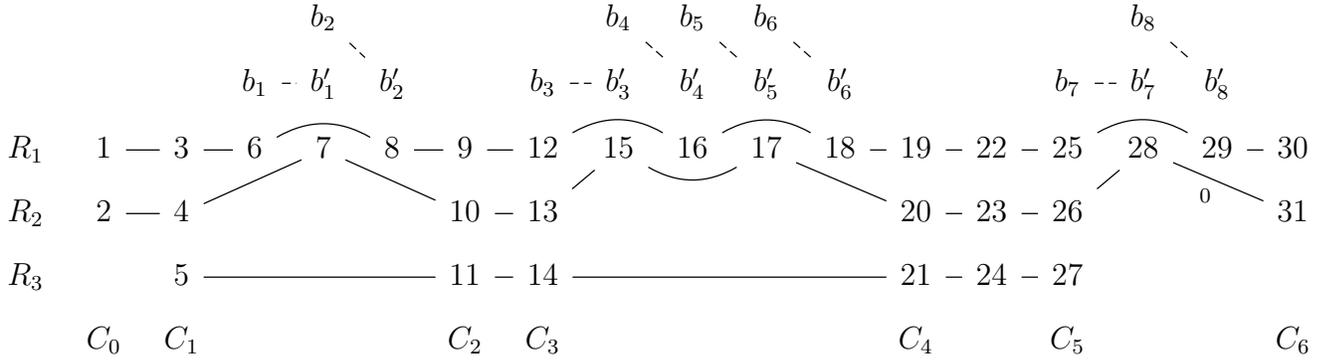
\begin{figure}[H]
\begin{tikzcd}[row sep=tiny, 
column sep = 0.5 em]
&&&&b_2\arrow[-,dashed,rd]&&&&b_4\arrow[-,dashed,rd]&b_5\arrow[-,dashed,rd]&b_6\arrow[-,dashed,rd]&&&&&b_8\arrow[-,dashed,rd]&&\\
&&&b_1\arrow[-,dashed,r]&b_1'&b'_2&&b_3\arrow[-,dashed,r]&b'_3&b'_4&b'_5&b'_6&&&b_7\arrow[-,dashed,r]&b'_7&b'_8&\\
R_1&1 \arrow[-,r] & 3 \arrow[-,r] &6\arrow[-,rr, bend left] &7 &8 \arrow[-,r] &9 \arrow[-,r] &12 \arrow[-,rr, bend left] &15\arrow[-,rr,bend right,swap] &16 \arrow[-,rr, bend left] &17 &18 \arrow[-,r] &19 \arrow[-,r] &22 \arrow[-,r] &25 \arrow[-,rr, bend left] &28 &29 \arrow[-,r] &30\\ 
R_2&2\arrow[-,r]&4\arrow[-,rru]&&&&10 \arrow[-,llu]\arrow[-,r]&13\arrow[-,ru]&&&&&20\arrow[-,llu]\arrow[-,r]&23\arrow[-,r]&26\arrow[-,ru]&&&\arrow[-,llu,"0"]31\\
R_3&&5\arrow[-,rrrr]&&&&11\arrow[-,r] &14\arrow[-,rrrrr] &&&&&21\arrow[-,r] &24\arrow[-,r] &27\\
&C_0&C_1&&&&C_2&C_3&&&&&C_4&&C_5&&&C_6
\end{tikzcd}
\caption{The gated lines after stage 1 are those joining the pair $(b_i,b_i'):i=1,2,\cdots,8$ the new lines obtained at stage following \ref{4.4.9} are drawn. Note that $P_1^2$ is satisfied.}
\end{figure}

\textbf{Example.} Let $\mathscr T$ be defined by the array $(3,2,1,1,2,3)$, for the step 3 one has a gating in each row.
\begin{figure}[H]
\begin{tikzcd}[row sep=tiny, 
column sep = 1 em]
1\arrow[-,r]&4\arrow[-,r]&6\arrow[-,r]&7\arrow[-,r]&8\arrow[-,r]&10\\
2\arrow[-,r]&5\arrow[-,rrr]&&&9\arrow[-,r]&11\\
3\arrow[-,rrrrr]&&&&&12\\
&&&\arrow[-,"\text{step 1}"]&&&\\
\end{tikzcd} \hspace{1em}
\begin{tikzcd}[row sep=tiny, 
column sep = 1em]
1\arrow[-,r,"1"]&4\arrow[-,r,"1"]&6\arrow[-,r,"0"]&7\arrow[-,r,"1"]&8\arrow[-,r,"1"]&10\\
2\arrow[-,r,"1"]&5\arrow[-,rrr,"0"]&&&9\arrow[-,r,"1"]&11\\
3\arrow[-,rrrrr,"0"]&&&&&12\\
&&&\arrow[-,"\text{step 2}"]&&&\\
\end{tikzcd} \hspace{1em}
\begin{tikzcd}[row sep=tiny, 
column sep = 1 em]
1\arrow[-,r,"1"]&4\arrow[-,r,"1"]&6\arrow[-,rrd,"0" ]&7\arrow[-,r,"1"]&8\arrow[-,r,"1"]&10\\
2\arrow[-,r,"1"]&5\arrow[-,rru,"1"]&&&9\arrow[-,r,"1"]&11\\
3\arrow[-,rrrrr,"0" ]&&&&&12\\
&&&\arrow[-,"\text{stage 3}"]&&&\\
\end{tikzcd} \hspace{1 em}
\begin{tikzcd}[row sep=tiny, 
column sep = 1em]
1\arrow[-,r,"1"]&4\arrow[-,r,"1"]&6\arrow[-,rrrdd,"0" ]&7\arrow[-,r,"1"]&8\arrow[-,r,"1"]&10\\
2\arrow[-,r,"1"]&5\arrow[-,rru,"1"]&&&9\arrow[-,r,"1"]&11\\
3\arrow[-,rrrru,"1"]&&&&&12\\
&&&\arrow[-,"\text{stage 3}"]&&&\\
\end{tikzcd}\caption{On the last diagram there are gated lines (not shown) between $(6,7)$ and $(6,9)$.}
\end{figure}
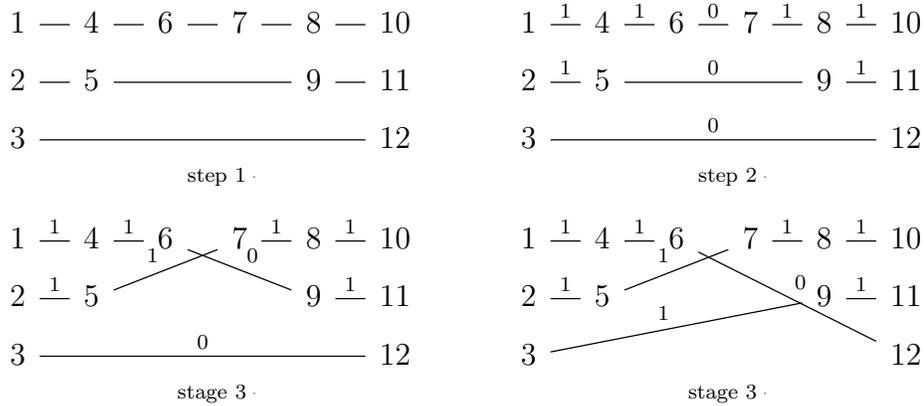

\subsection {}\label {4.6}
We now establish the assertion of the first paragraph of \ref {3.4.8}.  In this we are assuming $\mathfrak g$ is of type $A_{\ell-1}$ and we retain the notation of Section \ref {4}.

Let $\mathfrak p_{\pi_1},\mathfrak p_{\pi_2}$ be parabolics whose Levi factors are conjugate. We may represent them by diagrams $\mathscr D_1,\mathscr D_2$ whose columns have the same length up to permutation. Insert $\{1,2,\ldots,\ell\}$ into their boxes in the manner described in \ref {4.1.4}.  The resulting tableaux $\mathscr T_1,\mathscr T_2$ represent their nilradicals $\mathfrak m_{\pi_1},\mathfrak m_{\pi_2}$.

To prove the required assertion we only have to consider the case where we interchange two adjacent columns $C_2,C_3$ of heights $n \neq m$ and there exists a third column $C_1$, neighbouring to $C_2$ of height $n$.  We can assume that $C_1$ lies to the left of $C_2$.

Then the numbering in $\mathscr T_1$ is given by taking $C_1$ (resp. $C_2,C_3$) to have entries $\{i+1,i+2,\ldots,i+n\}$ (resp. $\{k+1,k+2,\ldots,k+n\},\{k+n+1,\dots, k+n+m\}$) down the rows, with $1\leq i+1, k+n+m \leq \ell$ and $i+n\leq k$.

Now let $w$ be the operation that subtracts $n$ from every entry of $C_2$ and adds $m$ to every entry of $C_2$ leaving all the other entries fixed.  It obviously belongs to the symmetric group $S_\ell$ which we recall is just the Weyl group in this case.  Again it clearly sends $\mathscr T_1$ to $\mathscr T_2$.


The ideal of definition of the hypersurface orbital variety $\varpi_1$ (resp. $\varpi_2$) in $\mathfrak m_1$ (resp. $\mathfrak m_2$) defined by the above pair of neighbouring columns of height is the Benlolo-Sanderson minor whose weight is just $\varpi_1=\sum_{j=1}^n \varepsilon_{i+j}-\varepsilon_{k+n+1-j}$ (resp. $\varpi_2=\sum_{j=1}^n \varepsilon_{i+j}-\varepsilon_{k+n+m+1-j}$), in the Bourbaki notation \cite [Tables]{Bo}.

\begin {lemma} One has $w\pi_1=\pi_2,w\varpi_1=\varpi_2$.  In particular $wr_{\mathscr V_1}=r_{\mathscr V_2}$.
\end {lemma}

\begin {proof} It is clear that only the elements of $\pi$ which come from $C_1,C_2$ are changed by $w$. Those coming from $C_1$ are unchanged because the entries in $C_1$ are unchanged.  Those coming from $C_2$ form the set $\{\varepsilon_{k+j}-\varepsilon_{k+j+1}\}_{j=1}^{n-1}$, which under $w$ becomes the set $\{\varepsilon_{k+j+m}-\varepsilon_{k+j+1+m}\}_{j=1}^{n-1}$.  These in turn are the elements of $\pi_2$ coming from $C_2$ in $\mathscr T_2$.  This proves the first part.  For the second part we note that $\varepsilon_{i+j}$ remains unchanged whilst $\varepsilon_{k+n+1-j}$ becomes $\varepsilon_{k+n+m+1-j}$, for all $j\in [1,n-1]$.  This proves the second part. Together they imply the third part.
\end {proof}

%
%
%
%
%

\section{The Number of Invariant Generators}\label{5}

 \subsection {}\label{5.1}

Fix a diagram $\mathscr D$ and let $g$ be the number of its pairs of neighbouring columns.

Define the algebra morphism $\varphi$ as in \ref {4.2.5} using the assignment of lines and labels on the lines given by Step 2.

The Benlolo-Sanderson construction gives $g$ invariant polynomials in $\mathbb C[\mathfrak m]^{\mathfrak p'}$.  By Lemma \ref {4.2.5}(ii) these polynomials must be algebraically independent since their images under $\varphi$ have this property.

Since $\mathbb C[\mathfrak m]^{\mathfrak p'}$ is a domain whose image has GK dimension $g$, it follows that $\varphi$ can admit a non-zero kernel only if the GK dimension of $\mathbb C[\mathfrak m]^{\mathfrak p'}$ is strictly greater than $g$, which translates in the present situation to saying that $\mathbb C[\mathfrak m]^{\mathfrak p'}$ is a polynomial algebra on $>g$ generators.

By \cite {JM} the ideal of definition of a hypersurface orbital variety in $\mathfrak m$ is the zero locus of an irreducible $P$ semi-invariant polynomial in $C[\mathfrak m]$.  Moreover the converse holds and so the number of such polynomials is just the number of hypersurface orbital varieties in $\mathfrak m$.  By \cite {JM}, this number is exactly $g$.

On the other hand $\mathbb C[\mathfrak m]^{\mathfrak p'}$ is generated by irreducible polynomials, which are weight vectors, as a polynomial algebra.  We conclude that $\mathbb C[\mathfrak m]^{\mathfrak p'}$ has GK dimension $g$ and so $\varphi$ is injective.

Here we wish to give a second proof of injectivity not using the counting of hypersurface orbital varieties.  For this it is enough to show the

\begin {prop}  $P'(e+V)$ is dense in $\mathfrak m$.
\end {prop}

\textbf{N.B.}  Notice that the pair $e+V$ is given by Step $2$.

 \subsection {Proof of the Proposition}\label{5.2}

 \subsubsection {}\label{5.2.1}

 Let $\mathscr H$ denote the set of all horizontal lines in $\mathscr T$.  For all $i \in \mathbb N^+$, let $m_i$ denote the number of boxes in the $i^{th}$ row of $\mathscr T$.

  Given blocks $b_1,b_2$ on the same row with $b_1$ to the left of $b_2$, let $\ell_{b_1,b_2}$ be the horizontal line joining them and $x_{b_1,b_2}$ the corresponding element in $\mathfrak m$ and $\alpha_{b_1,b_2}:=\varpi _{b,b'}$ its weight, which is in fact a root.

  Let $(. \ , \ .)$ denote the Cartan inner product on $\mathfrak h^*$.

 \begin {lemma}  The roots defined by the elements of $\mathscr H$ are linearly independent and form a system of type $\prod_{i \in \mathbb N^+}A_{m_i-1}$.
 \end {lemma}

 \begin {proof}  The entries of the boxes are pairwise distinct, so for horizontal lines on distinct rows $x_{b_1,b_2}$ commutes with $x_{b_3,b_4}$ and its transpose. Hence  $(\alpha_{b_1,b_2},\alpha_{b_3,b_4})=0$.  On the other hand the weights of the horizontal lines joining adjacent boxes on the same line have distinct supports, so are linearly independent.  The proves the first part of the lemma.

 For horizontal lines joining adjacent boxes on a given row, a scalar product is only non-zero when a box is shared and then it equals $-1$.   Consequently the composite lines on $R_i$ which pass through $m_i$ boxes, form a system of type $A_{m_i-1}$.  Hence the second part.
 \end {proof}

 \subsubsection {}\label{5.2.2}

  Let $\mathscr H_1$ (resp. $\mathscr H_0$) denote the subset of $\mathscr H$ of lines carrying $1$ (resp. $0$).

 Let $\Delta$ be the set of all roots in $\mathfrak g$.   Let $\Pi$ be the set of roots defined by $\mathscr H$.  Let $\mathfrak h^\perp$ be the orthogonal of $\Pi$ in $\mathfrak h$.

 It follows from Lemma \ref {5.2.1} that $\mathbb Z \Pi \cap \Delta$ is a root subsystem of $\Delta$.  Moreover we may choose a complement $\mathfrak h_\Pi$ of $\mathfrak h^\perp$ as a Cartan subalgebra of the corresponding semisimple Lie subalgebra $\mathfrak s$ of $\mathfrak {sl}(n)$.

 Set $\Delta_\Pi:=\Delta \cap \mathbb Z \Pi$.

 \begin {lemma} There exists $h \in \mathfrak h$ such that the $-1$ eigensubspace of $h$ on $\mathfrak {sl}(n)$ is that generated by the $x_\alpha : \alpha \in \Pi$ and the $0$ eigenspace is $\mathfrak h^\perp$.
 \end {lemma}

 \begin {proof} By Lemma \ref {5.2.1} there exists $h_1 \in \mathfrak h_\Pi$ which takes the value $-1$ on the elements of $\Pi$.   Then the eigenvalues of $h$ on $\Delta_\Pi$ are non-zero integers.

 Choose $h_2 \in \mathfrak h^\perp$ in general position and set $h=h_1+h_2$.   Then $h$ takes integer values on a root in $\mathfrak {sl}(n)$ if and only if it belongs $\Delta_\Pi$. Hence the lemma.
 \end {proof}

 \subsubsection {}\label{5.2.3}

 Let $h$ be as in the conclusion of Lemma \ref {5.2.2}. Given $V$ an semisimple $h$ module. let $V_0$ denote its zero $h$ subspace and $V_{\neq 0}$ its $h$ stable complement.

 \begin {cor}

 \

 (i). $\mathfrak p'=\mathfrak p'_{\neq 0} + \mathfrak h'$,

  \

  (ii). $\mathfrak p=\mathfrak p_{\neq 0} + \mathfrak h$,

  \

  (iii). $\mathfrak p'_{\neq 0} = \mathfrak p_{\neq 0}$.

  \end {cor}

  \subsubsection {}\label{5.2.4}

 Let $E$ (resp. $V$) be the linear span of the root vectors corresponding to $\mathscr H_1$ (resp. $\mathscr H_0$).

 Choose $e \in E, v \in V$ in general position (say with coefficient everywhere being $1$.  Set $\hat{e}=e+v$.  By Lemma \ref {5.2.1}, the action of $\mathfrak h$ on $\hat{e}$ generates $E+V$, whilst by Proposition \ref {4.3.5}, the action of $\mathfrak h_\Pi$ on $e$ generates $E$.

 \begin {lemma}  $\mathfrak p'.(e+v)+V=\mathfrak m$.
 \end {lemma}

 \begin {proof}

 Through the corollary and the above observation we obtain $\mathfrak p.(e+v)= \mathfrak p_{\neq 0}.(e+v)+ \mathfrak h.(e+v)=\mathfrak p'_{\neq 0}.(e+v)+ E + V =\mathfrak p'.(e+v)+V$.  Yet after Ringel et al \cite {BHRR} the left hand side equals $\mathfrak m$.

 \end {proof}

\subsubsection {}\label{5.2.5}

That the conclusion of lemma \ref {5.2.4} implies the conclusion of the proposition is fairly standard =  see for example \cite [bottom of p. 262 to top of p. 263] {D};  but we shall write out the details for completeness.  Let $U$ be the subspace $\mathfrak p'.(e+v)$.  Then $P'.(e+v)$ is dense in $U$ since $\mathfrak p'$ is the Lie algebra of $P'$.

Take  $\xi \in \mathfrak m^*$, such that \newline $(e+v +v')(p'\xi)=0$ for all $v' \in V$ and $p' \in P'$.  This must hold in particular for $v'=0$. Thus $(e+v)(p'\xi)=0$ for all $p' \in P'$ and $v'(\xi)=0$, for all $v' \in V$. By the first (resp. second) part  $\xi$ vanishes on $U$ (resp. $V$). Hence $\xi=0$ proving the proposition.

%
%
%

\subsubsection {}\label{5.2.6}

Now independent of \cite {JM} we obtain

\begin {cor} $S(\mathfrak m^*)^{\mathfrak p'}$ is a polynomial algebra with the number of generators being the number of hypersurface orbital varieties in $\mathfrak m$.
\end {cor}

\textbf{Remark.} The generators are the Benlolo-Sanderson invariants described in \cite {BS} and recalled in \ref {4.1.6}.

 \subsection {Irreducibility of the Benlolo-Sanderson Invariants.}\label{5.3}

 A delicate point in \cite {BS} was to show that the Benlolo-Sanderson invariants are irreducible.  Indeed this can fail if we take the corresponding minor between two columns $C,C'$ of equal height which are not neighbouring, the resulting invariant being just the product of the invariants defined by the successive neighbouring columns between $C,C'$.

 In \cite {JM}, three proofs of irreducibility were given.  Two of these proofs relied on the irreducibility of the associated variety of a simple highest weight module with integral highest weight in type $A$.  Recently G. Williamson \cite {W} has shown that this can fail. The third proof was purely combinatorial but arduous.


 \begin {cor}  The Benlolo-Sanderson invariants are irreducible.
 \end {cor}

 \begin {proof}
 Let us recall if $D$ is a domain over a field of characteristic zero and $A$ an algebra generated by nilpotent derivations of $D$, then a factor of an element of $D^A$ again lies in $D^A$.  In particular the factors of a Benlolo-Sanderson invariant again lie in  $\mathbb C[\mathfrak m]^{\mathfrak p'}$.

 By Lemma \ref {4.2.5}(iii) and our assumption, the restriction map $\varphi$ sends a Benlolo-Sanderson invariant to a co-ordinate function on $V$, hence it is injective on the algebra they generate.  By Corollary \ref {5.2.6} this algebra is the \textbf{whole} invariant ring $\mathbb C[\mathfrak m]^{\mathfrak p'}$.  Thus the restriction map is an algebra isomorphism.  Then a  Benlolo-Sanderson invariant, having image a co-ordinate function, must be irreducible.
 \end {proof}

  \subsection {The Nilfibre.}\label{5.4}

  \subsubsection {} \label{5.4.1}
  Obviously $\mathbb C[\mathfrak m]^{\mathfrak p'}$ inherits the gradation  of $\mathbb C[\mathfrak m]$ defined by degree of homogeneous polynomials. Let $\mathbb C[\mathfrak m]^{\mathfrak p'}_+$ denote the augmentation of $\mathbb C[\mathfrak m]^{\mathfrak p'}$ (that it say the subspace spanned by homogeneous invariants of positive degree) let $\mathscr N$ denote its zero locus.  It is called the nilfibre for the action of $P'$ on $\mathfrak m$.

  Obviously $\mathscr N$ is $P$ invariant. A pleasant extension of Richardson's theorem would be that $\mathscr N$ admits a dense $P$ orbit (necessarily of codimension $g$).  However this would imply that $\mathscr N$ is irreducible: but this can fail.

   One can ask if a component $\mathscr C$ of $\mathscr N$ admits a dense $P$ orbit.  This would be of great interest.  Indeed let $e$ be a representative in such an orbit. Since $\mathscr N$ is a cone so is $\mathscr N$, so in particular $ac \in \mathscr C$ for every non-zero scalar $a$. Consequently there exists $p \in P$ such that $p.c=ac$ and hence an element $h \in \mathfrak p$ such that $h.c=c$.  Since $P$ is an algebraic group we can assume $h$ to be ad-semisimple without loss of generality and hence belong to a Cartan subalgebra, which we can assume to be our fixed Cartan subalgebra $\mathfrak h$.

   Now let $V$ be an $h$ stable complement to $\mathfrak p.e$ in $\mathfrak m$.  One can ask if $e+V$ is a Weierstrass section and if all Weierstrass sections so arise.

   \textbf{Example.}  Let $\mathscr T$ be defined by the array $(2,1,1,2)$.  Then $\mathbb C[\mathfrak m]^{\mathfrak p'}$ is generated by $x_{3,4}$ and the second Benlolo-Sanderson invariant which is an irreducible polynomial of degree $4$.  Mod $x_{3,4}$ it is a product of two $2\times 2$ minors.  Thus $\mathscr N$ is not irreducible but equidimensional of codimension $2$ in $\mathfrak m$.

  \subsubsection {} \label{5.4.2}

    By \ref {4.4}, we can assume that properties $P_1,P_2$ hold. Then by Lemma \ref {4.2.5} we can choose a linear subvariety $e+V$ of $\mathfrak m$, so that a Benlolo-Sanderson invariant restricts to a co-ordinate function on $V$.  Define $E,V$ as in \ref {5.2.4}.

   \begin {lemma}  One has $E \subset \mathscr N$.
   \end {lemma}

   \begin {proof} Indeed by $(P_2)$ a Benlolo-Sanderson invariant restricted to $E+V$ has just one non-vanishing term in its standard expansion. All the factors of this term are elements coming from $E$, except one which comes from $V$. Hence this invariant vanishes on $E$.
   \end {proof}

   \textbf{Remark.}  The element $e$  given in the construction of \ref {4.2.4} does not generally give an element which generates a $P'$ orbit in $\mathfrak m$ of the maximal dimension $\dim \mathfrak m -g$.  Yet it might generate a $P$ orbit of this dimension in $\mathscr N$, though even this may fail.

  \textbf{ Example 1.}  Let $\mathscr T$ be given by the array $2,1,1,2$.  Here our construction gives $e=x_{1,3}+x_{2,4}+x_{4,5}, V=\mathbb Cx_{3,4}+\mathbb Cx_{3,6}$.  In this case $P'.e$ (resp. $P.e$) has codimension $3$ (resp. $2$) in $\mathfrak m$, whilst $g=2$.
  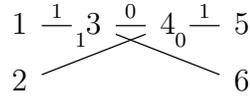
\begin{figure}[H]
 \begin{tikzcd}[row sep=tiny, 
column sep = 1 em]
1\arrow[-,r,"1"]&3\arrow[-,r,"0"]\arrow[-,rrd,"0"]&4\arrow[-,r,"1"]&5\\
2\arrow[-,rru,"1"]&&&6\\
\end{tikzcd}
\caption{Labeling of the line after step 3, the lines labeled  $0$ represent $V$ while the label lines labeled  $1$ represent $e$. }
\end{figure}

  \textbf{Example 2.}  Let $\mathscr T$ be given by the array $1,2,2,1$.  Here our construction gives $e=x_{1,2}+x_{2,4}, V=\mathbb C x_{4,6} +\mathbb Cx_{3,5}$. In this case even $P.e$ has codimension $>2=g$.  This may be remedied by adding the element $x_{5,6}$.  Then $E:=\mathbb C x_{1,2}+\mathbb C x_{2,4}+\mathbb C x_{5,6}$.  In this case $E$ still lie in $\mathscr N$, whilst $P.(e+x_{3,5})$ has codimension $2$ in $\mathfrak m$. In this case the Benlolo-Sanderson invariants are
  $$x_{2,4}x_{3,5}-x_{2,5}x_{3,4}, \quad
  x_{1,2}x_{2,4}x_{4,6}+x_{1,3}x_{3,4}x_{4,6}+
  x_{1,2}x_{2,5}x_{5,6}+x_{1,3}x_{3,5}x_{5,6}.$$

  In this case we believe $\mathscr N$ to be irreducible.
  
  \begin{figure}[H]
\begin{tikzcd}[row sep=tiny, 
column sep = 1 em]
1\arrow[-,r,"1"]&2\arrow[-,r,"1"]&4\arrow[-,r,"0"]&6\\
&3\arrow[-,r,"0"]&5\arrow[-,ru]&\\
\end{tikzcd}
\caption{Labeling of the line after step 3, the lines labeled  $0$ represent $V$ and the rest of the lines represent $E$ }
\end{figure}

  \section {The Main Theorem}\label {6}

  \subsection {}\label {6.1}

  Let $\mathfrak g$ be a simple Lie algebra of type $A_{n-1}$  .Fix a parabolic subalgebra $\mathfrak p_{\pi'}$ of $\mathfrak g$, $\mathfrak p_{\pi'}'$ it derived algebra  and $\mathfrak m_{\pi'}$ be its nilradical.  Let $P'_{\pi'}$ be the connected algebraic subgroup of $G$ with Lie algebra  $\mathfrak p_{\pi'}'$.

  \subsection {}\label {6.1}

  Define the linear subvariety $e+V$ of $\mathfrak m_{\pi'}$ as in \ref {4.2} satisfying $(P_1),(P_2)$.  Its existence was established in \ref {4.4}.  Recall the definition of a Weierstrass section given in \ref {1.1}.

  \begin {thm}  $e+V$ is a Weierstrass section for the action of $P'_{\pi'}$ on $\mathfrak m_{\pi'}$.
  \end {thm}

  \begin {proof}  Since $(P_1),(P_2)$, hold, injectivity of the restriction map $\varphi$ follows from Lemma \ref {4.2.5}(iii).  Then injectivity follows from Corollary \ref {5.2.6}.
  \end {proof}

 \section {Index of Notation }\label {6}

 Symbols used frequently are given below in the order in which they appear.

 \

 \ref {1}.    \quad  \quad    $\mathbb C, [1,n]$.

 \ref {1.1}.  \quad  $\varphi$.

 \ref {1.3.1}. $\mathscr N$.

 \ref {2.1}. \quad $P,G,\mathfrak g, \mathfrak p, \mathfrak r, \mathfrak m, \mathfrak h, \pi, \mathfrak p_{\pi'}, \mathfrak b, \mathfrak n, \mathfrak p', P'$.

 \ref {2.3.1}. $\mathscr O$.

 \ref {2.3.4}. $N_{\pi,\pi'}$.

 \ref {3.4.4}. $\Delta, \Delta^+, r_\mathscr V,p_\mathscr V$.

 \ref {3.4.5}. $\Pi',\Pi'_-$.

 \ref {4.1.1}. $\mathscr D, ht \mathscr D, C_v$.

 \ref {4.1.3}. $M_{v,v'}$.

 \ref {4.1.4}. $\mathscr T$.

 \ref {4.1.5}. $x_{i,j}$.

 \ref {4.1.6}. $M^t,d_{u,v}$.

 \ref {4.1.7}. $b<b'$.

 \ref {4.2.4}. $\ell_{b,b'},\varpi_{b,b'}$.

 \ref {4.4.1}. $R^i$.

 \ref {5.2.1}. $\mathscr H$.

\

\end{document}